\documentclass[12pt]{amsart}
\usepackage[mathcal]{eucal}
\usepackage{amssymb, graphicx, labelfig}

\newtheorem{thm}{Theorem}
\newtheorem{prop}[thm]{Proposition}
\newtheorem{lem}[thm]{Lemma}

\newtheorem{cor}[thm]{Corollary}

\theoremstyle{remark}

\theoremstyle{definition}
\newtheorem{defi}[thm]{Definition}

\newcommand{\col}{\kern -3pt :}
\newcommand{\C}{\mathbb C}
\newcommand{\R}{\mathbb R}
\newcommand{\Z}{\mathbb Z}

\newcommand{\HH}{\mathbb H}
\newcommand{\PP}{\mathbb P}
\newcommand{\Id}{\mathrm{Id}}
\newcommand{\End}{\mathrm{End}}
\newcommand{\Isom}{\mathrm{Isom}^+}
\newcommand{\Diff}{\mathrm{Diff}}

\newcommand{\T}{\mathcal T}
\newcommand{\U}{\mathcal U}
\newcommand{\K}{\mathcal K}
\newcommand{\PSL}{\mathrm{PSL}}
\newcommand{\E}{\mathrm{e}}
\newcommand{\I}{\mathrm{i}}

\renewcommand{\geq}{\geqslant}
\renewcommand{\phi}{\varphi}

\title[Representations
of the quantum Teichm\"uller space]
{Local representations\\
of the quantum Teichm\"uller space}
\author{Hua Bai}
\address{Department
of Mathematics, University of Georgia, Athens GA~, U.S.A.}
\email{huabai@math.uga.edu}

\author{Francis Bonahon}
\address {Department
of Mathematics,  University of
Southern California, Los Angeles
CA~90089-2532, U.S.A.}
\email{fbonahon@math.usc.edu}

\author{Xiaobo Liu}
\address{Department
of Mathematics, Columbia University, New York NY~, U.S.A.}
\email{xiaoboli@math.columbia.edu}

\date{\today}

\begin{document}

\begin{abstract}
We introduce a certain type of representations for the quantum Teichm\"uller space of a punctured surface, which we call local representations. We show that, up to finitely many choices, these purely algebraic representations are classified by classical geometric data. We also investigate the family of intertwining operators associated to such a representations. In particular, we use these intertwiners to construct a natural fiber bundle over the Teichm\"uller space and its quotient under the action of the mapping class group. This construction also offers a convenient framework to exhibit invariants of surface diffeomorphisms. 

\end{abstract}

\maketitle

The classical Teichm\"uller space $\T(S)$ of a surface $S$ is the space of complex structures on the surface and, as such, is ubiquitous in mathematics. When $S$ has negative Euler characteristic, the Uniformization Theorem also identifies $\T(S)$ to the space of complete hyperbolic metrics on $S$, which embeds $\T(S)$ in the space of group homomorphisms from the fundamental group $\pi_1(S)$ to the isometry group $\Isom(\HH^2) = \PSL_2(\R)$. This gives the Teichm\"uller space a more algebraic flavor.

For a punctured surface $S$, the quantum Teichm\"uller space $\T_S^q$ is a non-commutative deformation of the algebra of rational functions on the classical Teichm\"uller space $\T(S)$, depending on a parameter $q\in \C^*$. This is a purely algebraic object, completely determined by  the combinatorics of the Harer-Penner complex of all ideal cell decompositions of the surface. It was originally introduced by V.~Fock and L.~Chekhov \cite {Fok, CF} and, independently, by R.~Kashaev \cite{Kash1}. We are here using a variation of their original approach, the exponential version of the quantum Teichm\"uller space, which is better suited to mathematical applications; see \cite{BonLiu, Liu1}.
The main advantage of this exponential version is that, when $q$ is a root of unity, it admits an interesting finite dimensional representation theory. Irreducible (and suitably defined) representations of the quantum Teichm\"uller space were investigated in \cite{BonLiu}.

Irreducible representations clearly were the natural object to look at in a first analysis. In this paper, we introduce and investigate a different type of representations, called local representations. These are at first less natural, but they have a few advantages over irreducible representations. One of them is that they are simpler to analyze. However, their main interest is that they are related to other mathematical objects, such as Kashaev's $6j$--symbols \cite{Kash1} and the invariants of knots and 3--dimensional manifolds subsequently developed by Kashaev \cite{Kash2, Kash3}, Baseilhac and Benedetti \cite{BasBen02, BasBen03, BasBen05, BasBen06}. See \cite{Bai2, Bon10} for these connections.

Altogether, the theory of local representations  of the quantum Teichm\"uller space closely shadows  the analysis of irreducible representations in \cite{BonLiu}. For instance, our main result is the following theorem, proved as Theorem~\ref{thm:LocalRepGroupHom}
 in \S\ref{sect:Intertwiners}, which states that a local representation is classified by certain classical geometric data, plus a choice of some root of unity. Comparing this statement with the very similar Theorems~2 and 3 of \cite{BonLiu}, the reader will notice that the assertion is here somewhat simpler to state, and that its proof is also simpler. In particular, it does not use the fine analysis of the algebraic structure of the Chekhov-Fock algebra of \cite[\S\S2--3]{BonLiu}

\begin{thm}
\label{thm:LocalRepGroupHom}
Let $S$ be a punctured surface obtained by removing $p\geq 1$ points from a closed oriented surface $\bar S$. 
Let $q^2$ be a primitive $N$--th root of $q^N=(-1)^{N+1}$ (for instance $q =- \E^{\I \pi/N}$). Then, a local representation of the quantum Teichm\"uller space of  $S$ is classified by: 

\begin{enumerate}

\item a group homomorphism $r\col \pi_1(S) \to \Isom (\HH^3)$ from the fundamental group of $S$ to the isometry group of the $3$--dimensional hyperbolic space $\HH^3$; 
\item for each peripheral subgroup $\pi$ of the fundamental group $\pi_1(S)$ (corresponding to one of the punctures), the choice of a point $\xi_\pi$ of the sphere at infinity $\partial _\infty \HH^3$ that is fixed by $r(\pi)$;

\item the choice of an $N$--th root  $h\in \C^*$ for a certain number $h_r$ associated to $r$ and to the points $\xi_\pi$. 
\end{enumerate}
\end{thm}

Not every group homomorphism $r\col \pi_1(S) \to \Isom (\HH^3)$, enhanced by the data of the fixed points $\xi_\pi\in \partial _\infty \HH^3$, is associated in this way to a local representation of the quantum Teichm\"uller psace. One needs in addition that this data be well-behaved with respect to ideal triangulations of $S$. However, this condition is automatically satisfied if $r$ is injective, which occurs in most cases of interest. See \S\ref{sect:Pleated} for precise statements.

A local representation of the quantum Teichm\"uller space is actually a whole family of algebra homomorphisms $\rho_\lambda\col \T_\lambda^q \to \End (V_\lambda)$, where $\lambda$ ranges over all ideal triangulations of the surface, where the \emph{Chekhov-Fock algebra} $\T_\lambda^q$ is a certain algebra associated to the combinatorics of $\lambda$, and where $\End(V_\lambda)$ is the algebra of endomorphisms of a finite-dimensional vector space $V_\lambda$. Up to a few technicalities, the quantum Teichm\"uller space is in fact the family of such algebras $\T_\lambda^q$ where, to make the construction independent of a choice of ideal triangulation, we identify any two Chekhov-Fock algebras $\T_\lambda^q$ and $\T_{\lambda'}^q$ by a certain \emph{Chekhov-Fock coordinate change isomorphism} $\Phi_{\lambda\lambda'}^q\col \T_{\lambda'}^q \to \T_\lambda^q$. The homomorphisms $\rho_\lambda$ are then required to be compatible in the sense that, for any two ideal triangulations $\lambda$, $\lambda'$, the algebra homomorphism $\rho_\lambda\circ \Phi_{\lambda\lambda'}^q$ is isomorphic to $\rho_{\lambda'}$ by an isomorphism $L^\rho_{\lambda\lambda'}\col V_{\lambda'}\to V_\lambda$. 

These isomorphisms $L^\rho_{\lambda\lambda'}$, called \emph{intertwining operators}, play a fundamental role for applications. In Theorem~\ref{thm:IntertwiningUnique} of \S\ref{sect:Intertwiners}, we show that they are uniquely determined up to scalar multiplication, provided that we impose on them two naturality conditions involving all possible ideal triangulations (the Composition Relation) and all possible surfaces (the Fusion Relation). 

In \S\ref{sect:Kashaev}, we investigate a construction which was suggested to us by Rinat Kashaev. We use the intertwining operators as transition functions to construct a fiber bundle, with fiber a projective space, over the classical Teichm\"uller space or, more generally, over the space of quasifuchsian hyperbolic metrics on the product of $S\times \R$ and its closure. The construction is invariant under the action of the mapping class group, so that it descends to a fiber bundle over the moduli space. It also provides a natural framework to construct invariants of pseudo-Anosov diffeomorphisms of the surface $S$, similar to those defined in~\cite{BonLiu}. 

We should also mention that some of the ideas of this paper, including the fusion construction introduced in \S\ref{subsect:Fusing}, can be used to simplify some of the arguments of \cite{BonLiu}. 

\medskip
\noindent\textbf{Acknowledgement:} It is a pleasure to thank Rinat Kashaev for valuable conversations, including the idea that we exploited in \S\ref{sect:Kashaev}.

\section{The Chekhov-Fock algebra}
\label{sect:CheFock}

Although most of the applications will only involve punctured surfaces without boundary, it is convenient to include surfaces with boundary in the discussion. See the Fusion Relation of \S\ref{sect:QTS} and \S\ref{sect:Intertwiners}.

Let $S$ be an oriented  punctured surface with
boundary, obtained by removing finitely many points $v_1$,
$v_2$,
\dots, $v_p$ from a compact connected oriented surface
$\bar S$ of genus $g$ with (possibly empty) boundary. We require that each
component of $\partial \bar S$ contains at least one puncture $v_i$, that
there is at least one puncture, and that the Euler characteristic $\chi(S)$ of $S$ is negative. These topological restrictions, which exclude only a small number of  surfaces (assuming that there is at least one puncture), are
equivalent to the existence of an \emph{ideal triangulation} for $S$,
namely a triangulation of the closed surface
$\bar S$ whose vertex set is exactly
$\{ v_1,
\dots, v_p\}$. In particular, an ideal triangulation
$\lambda$  has $n=-3\chi(S)+2s$ edges and $m=-2\chi(S)+s$ faces. Its edges
provide $n$ arcs $\lambda_1$, \dots, $\lambda_n$ in $S$, going from
puncture to puncture, which decompose the surface $S$ into  $-2\chi(S)+s$
infinite triangles whose vertices sit `at infinity' at the punctures.
Note that $s$ of these $\lambda_i$ are just the boundary components
of $S$. 

We will occasionally need to consider disconnected surfaces,
in which case the above conditions will apply to each
component of these surfaces. 

Fix a non-zero complex number $q = \mathrm e^{\pi\mathrm i \hbar} \in \C^*$. The \emph{Chekhov-Fock algebra} associated to the ideal triangulation
$\lambda$ is the algebra $\T_\lambda^q$ defined by generators
$X_i^{\pm1}$ associated to the edges $\lambda_i$ of $\lambda$, and by
relations
$X_iX_j= q^{2\sigma{ij}}X_jX_i$ where the integers $\sigma_{ij}\in \{ 0, \pm1, \pm2\}$ are
defined as follows: if $a_{ij}$ is the number of angular sectors  delimited  by $\lambda_i$ and  $\lambda_j$ in the
faces of
$\lambda$, and with $\lambda_i$
coming first counterclockwise, then $\sigma_{ij} = a_{ij}-a_{ji}$. 

In practice, the elements of the Chekhov-Fock algebra $\mathcal
T_\lambda^q$ are Laurent polynomials in the variables $X_i$,
and are manipulated in the usual way except that their multiplication
uses the skew-commutativity relations $X_iX_j= q^{2\sigma{ij}}X_jX_i$. We
will sometimes write $\T_\lambda^q= \mathbb C [ X_1^{\pm1},\dots,
X_n^{\pm1}]_\lambda^q$ to reflect this fact. 

The product $X_1 X_2\dots X_n$ of all the generators is easily seen to be
central in $\T_\lambda^q$. However, we will prefer the
\emph{principal central element}
$$
H= q^{-\sum_{i<j} \sigma_{ij}} X_1 X_2\dots X_n
$$
which has better invariance properties. Indeed, the exponent of the $q$--factor is
specially designed to make $H$ independent of the indexing of the $X_i$.
This symmetrization is classical in physics, where it is known as the Weyl
quantum ordering. 

The center of 
$\T_\lambda^q$ is generated by $H$ and  by certain
elements associated to the punctures of $S$ plus, when
$q^2$ is a primitive $N$--root of unity, the powers
$X_i^N$ and in addition, when $N$ is even, a few additional generators. See for instance
\cite[\S3]{BonLiu}.

\section{The triangle algebra} 
\label{sect:triangle}

Consider the case where $S=T$ is a triangle with its vertices deleted, namely where $\bar S$ is a closed disk and there are 3
punctures $v_1$, $v_2$, $v_3$ on its boundary. In this case,
$\lambda$ has one face and three edges, all in the boundary of $\bar
S$. Index  the edges of
$\lambda$ so that $\lambda_1$, $\lambda_2$, $\lambda_3$ occur
clockwise in this order around $\partial\bar S$. Then the
Chekhov-Fock algebra is the \emph{triangle algebra} defined by
the generators
$X_1^{\pm1}$, $X_2^{\pm1}$, $X_3^{\pm1}$ and by the relations $X_iX_{i+1}
= q^2 X_{i+1}X_i$ for every $i$ (considering indices modulo 3). Note that
the principal central element is $H=q^{-1}X_1X_2X_3$. 

\subsection{Representations of the triangle
algebra} 
\label{sect:trianglerep}

We  consider
representations of the triangle algebra $\mathcal
T_T^q$, namely algebra homomorphisms $\rho\col
\T_T^q \to
\End(V)$ valued in the algebra of linear
endomorphisms of a complex vector space $V$. If we want
$V$ to be finite-dimensional, it is not hard to
see that we need $q^2$ to be a root of unity, namely that
$q^{2N}=1$ for some integer $N>0$. This implies that $q^N =
\pm 1$, but our analysis works best when $q^N = (-1)^{N+1}$. 

Consequently, we will henceforth assume that $q^N =
(-1)^{N+1}$ and that $N>0$ is minimum for this property. In
other words, we require that $-q$ be a primitive $N$--th
root of $-1$.

\begin{lem}
\label{lem:TriangleRep}

 Given an irreducible representation $\rho \col \T_T \to \End(V)$ of the triangle algebra, there exists non-zero complex numbers $x_1$, $x_2$, $x_3$, $h\in \C^*$ such that $\rho$ sends each power $X_i^N$ to $x_i \Id_V$, and sends $H$ to $ h \Id_V$. Two irreducible representations of the triangle algebra are isomorphic if and only if they are isomorphic if and only if they define the same numbers $x_1$, $x_2$, $x_3$, $h\in \C^*$ as above.

In addition, $h^N = x_1x_2x_3$, and any set of numbers $x_1$, $x_2$, $x_3$, $h\in \C^*$ with $h^N = x_1x_2x_3$ can be realized in this way by an irreducible representation of the triangle algebra.
\end{lem}

\begin{proof} The proof is completely elementary,
by analyzing the action of the $\rho(X_i)$ on
the eigenvectors of  $\rho(X_1)$. Indeed, one easily finds a basis $e_0$, $e_1$, \dots, $e_{N-1}$ for $V$ and numbers $y_1$, $y_2$, $y_3\in \C^*$ such that $\rho(X_1)=y_1q^{2i}e_i$, $\rho(X_2)(e_i) = y_2 e_{i+1}$ and $\rho(X_3) (e_i) = y_3 q^{1-2i}e_{i-1}$. The rest of the properties easily follow from this fact. 
\end{proof}

In particular, an irreducible representation  $\rho\col \T_T ^q\to
\End(V)$  is classified by  three complex numbers
$x_1$, $x_2$, $x_3$ associated to the sides of the triangle, and by
an
$N$--root
$h=\sqrt[N]{x_1x_2x_3}$ for their product. We will call $h$  the
\emph{central load} of the representation $\rho$.

\section{Local representations of the Chekhov-Fock
algebra}
\label{sect:LocalRep}

\subsection{Embedding the Chekhov-Fock algebra in a
tensor product}
\label{sect:EmbedCheFock}

 Let $\lambda$ be an ideal
triangulation of the punctured surface $S$ with
faces $T_1$,
\dots,
$T_m$ where $m=-2\chi(S)+s$. Each face
$T_j$ determines a triangle algebra $\T_{T_j}^q$, with generators
associated to the three sides of $T_j$. 

The Chekhov-Fock algebra $\T_\lambda^q$ has a natural embedding  into the tensor product algebra $\bigotimes_{j=1}^m \T_{T_j}^q = \T_{T_1}^q \otimes \dots \otimes \T_{T_m}^q$
defined as follows. If the generator $X_i$ of $\T_\lambda^q$ is
associated to the $i$--th edge $\lambda_i$ of $\lambda$, define
\begin{itemize}
\item $f(X_i) = X_{ji} \otimes X_{ki}$ if $\lambda_i$ separates two
distinct faces
$T_j$ and
$T_k$, and if
$X_{ji}\in \T_{T_{j}}^q$ and $X_{ki}\in \mathcal
T_{T_{k}}^q$ are the generators associated to the sides of
$T_j$ and $T_k$ corresponding to
$\lambda_i$;
\item $f(X_i) = q^{-1} X_{ji_1} X_{ji_2} =q X_{ji_2} X_{ji_1}$ if $\lambda_i$
corresponds to two sides of the same face $T_j$, if $X_{ji_1}$,
$X_{ji_2}\in\T_{T_j}^q$ are the generators associated to these two
sides, and if $X_{ji_1}$ is associated to the side that comes first when going
counterclockwise around their common vertex. 
\end{itemize}
By convention, when describing an element $Z_1 \otimes  \dots \otimes Z_m$ of $\T_{T_1}^q \otimes  \dots \otimes \T_{T_m}^q$, we omit in the tensor product those $Z_j$ that are equal to the identity element $1$ of $\T_{T_j}^q$.

\begin{lem}
There exists a (unique) injective algebra homomorphism 
$$f\col \T_\lambda^q \to
\T_{T_1}^q \otimes  \dots \otimes \T_{T_m}^q$$
 such that, for every generator $X_i$ of $\T_\lambda^q$, $f(X_i) $ is the element of $\bigotimes_{j=1}^m \T_{T_j}^q$ defined above.
\end{lem}

\begin{proof}
The fact that $f$ extends to an algebra homomorphism is a consequence of the definition of the coefficients $\sigma_{ij}$ occuring in the
skew-commutativity relations of $\T_\lambda^q$. The injectivity  immediately follows from the fact that, as a vector space, $\T_\lambda^q$ admits a natural basis consisting of all the monomials $X_1^{n_1} \dots X_m^{n_m}$ with $n_j\in\Z$. 
\end{proof}

We will henceforth consider
$\T_\lambda^q$ as a subalgebra of $\bigotimes_{j=1}^m \mathcal
T_{T_j}^q$.

\subsection{Local representations}
\label{sebsect:LocalRep}

Now, suppose that we are given an irreducible representation $\rho_j\col
\T_{T_j}^q \to \End (V_j)$ for every face $T_j$. The tensor
product $\bigotimes_{j=1}^m \rho_j \col \bigotimes_{j=1}^m \mathcal
T_{T_j}^q \to \End (\bigotimes_{j=1}^m V_j)$ restricts to a representation
$$
\rho\col \T_\lambda^q \to \End (V_1 \otimes \dots \otimes V_m).
$$

We would like to define a local representation of $\T_\lambda^q$ as
any representation obtained in this way. However, we need to be a little more careful in the case where an edge of the ideal triangulation $\lambda$ corresponds to two sides of the same face of the triangulation

\begin{defi}
A \emph{local representation}  of $\T_\lambda^q$ is an equivalence class of $m$--tuples $(\rho_1, \rho_2, \dots, \rho_m)$ where each $\rho_j\col
\T_{T_j}^q \to \End (V_j)$ is an irreducible representation of the Chekhov-Fock algebra $\T_{T_j}$ associated to the $j$--th face $T_j$ of $\lambda$, and where $(\rho_1, \rho_2, \dots, \rho_m)$ and $(\rho_1', \rho_2', \dots, \rho_m')$ are identified if and only if:
\begin{enumerate}
\item for every face $T_j$ of $\lambda$, the representations $\rho_j$ and $\rho_j'$ act on the same vector space $V_j=V_j'$;
\item for every edge $\lambda_i$ of $\lambda$,
\begin{enumerate}
\item  if $\lambda_i$ separates two
distinct faces $T_j$ and $T_k$ and if $X_{ji}\in \T_{T_{j}}^q$ and $X_{ki}\in \mathcal
T_{T_{k}}^q$ are the generators respectively associated to the sides of
$T_j$ and $T_k$ corresponding to
$\lambda_i$, there exists a number $a\in \C^*$ such that $\rho_j'(X_{ji} )= a\rho_j(X_{ji})$ and $\rho_j'(X_{ki}) = a^{-1}\rho_j(X_{ki})$;
\item  if $\lambda_i$
corresponds to two sides of the same face $T_j$ and if $X_{ji_1}$,
$X_{ji_2}\in\T_{T_j}^q$ are the generators associated to these two
sides, there exists a number $a\in \C^*$ such that $\rho_j'(X_{ji_1}) = a\rho_j(X_{ji_1})$ and $\rho_j'(X_{ji_2}) = a^{-1}\rho_j(X_{ji_2})$. 
\end{enumerate}
\end{enumerate}
\end{defi}

A local representation clearly defines a unique representation $\rho\col \T_\lambda^q \to \End (V_1 \otimes \dots \otimes V_m)$, where each vector space $V_j$ is associated to the $j$--th face $T_j$ of $\lambda$ and has dimension $N$. If every edge of $\lambda$ separates two distinct faces,  the representations $\rho_j\col
\T_{T_j}^q \to \End (V_j)$ can be easily recovered from $\rho$ up to the above equivalence relation. However, this is not true when one edge corresponds to two sides of the same face of $\lambda$. 

By abuse of notation, we will often refer to a local representation by mentioning only the representation  $\rho\col \T_\lambda^q \to \End (V_1 \otimes \dots \otimes V_m)$. This is no abuse if every edge of $\lambda$ belongs to two distinct faces. However, to deal with the special cases where this does not hold, we should always remember that a local representation actually involves the data of a family of irreducible representations $\rho_j\col
\T_{T_j}^q \to \End (V_j)$ associated to the faces of $\lambda$, and well-defined modulo the above equivalence relation. 

To explain the terminology note that, for every generator $X_i$ of $\T_\lambda^q$ associated to the $i$--th edge $\lambda_i$ of $\lambda$, its image $\rho(X_i)$ under a local representation $\rho$ is locally defined, in the sense that it depends only on data associated to the one or two faces of $\lambda$ that are adjacent to $\lambda_i$. 

\subsection{Classification of local
representations}
\label{sect:ClassifyLocalRep}

 Two local representations 
$
\rho\col \T_\lambda^q \to \End (V_1 \otimes \dots \otimes V_m)
$
and 
$
\rho'\col \T_\lambda^q \to \End (V_1' \otimes \dots \otimes V_m')
$ of the Chekhov-Fock algebra $\T_\lambda^q$
are \emph{isomorphic as local representations} if they can be realized by representations $\rho_j\col
\T_{T_j}^q \to \End (V_j)$ and $\rho_j'\col
\T_{T_j}^q \to \End (V_j')$ such that each $\rho_j$ is isomorphic to $\rho_j'$, namely such that there exists linear isomophisms
$L_j\col V_j '\to V_j$ such that 
$$
\rho_j'(X) = L_j ^{-1}\circ \rho_j(X) \circ  L_j
$$
for every $X\in \T_{T_j}^q$. 

\begin{lem}
\label{lem:LocalRepInvariants}
Let $
\rho\col \T_\lambda^q \to \End (V_1 \otimes \dots \otimes V_m)
$ be a local representation of the Chekhov-Fock algebra $\mathcal
T_{\lambda}^q$. If $X_i\in\T_\lambda^q$ is associated to the $i$--th
edge $\lambda_i$ of $\lambda$, then
$\rho(X_i^N) = x_i\,
\Id_{V_1\otimes\dots \otimes V_m}$ for some complex number $x_i \in
\C^*$. Similarly, there exists an  $h \in \C^*$ such that
$\rho(H) = h\, \Id_{V_1\otimes\dots \otimes V_m}$ for the principal
central element $H= q^{-\sum_{i<j} \sigma_{ij}} X_1 X_2\dots X_n$. In
addition,
$h^N = x_1
\dots x_n$. 
\end{lem}

\begin{proof} Lemma~\ref{lem:TriangleRep} proves
this for the triangle algebra. The general case
immediately follows.  Note that $H=H_1\dots H_m$
where $H_j$ denotes the principal central element
of the triangle algebra $\T_{T_j}^q$
associated to the face $T_j$ of $\lambda$. 
\end{proof}

In particular, Lemma~\ref{lem:LocalRepInvariants}
associates to the local representation $\rho$ a
system of locally defined edge weights
$x_i\in
\C^*$ for the ideal triangulation $\lambda$,
as well as a global invariant
$h\in
\C^*$. We call $h$ the \emph{central load} of the
representation $\rho$. 

\begin{prop}
\label{prop:ClassifyLocalRep}
Up to isomorphism of local representations,  the local
representation
$
\rho\col \T_\lambda^q \to  \End (V_1 \otimes \dots
\otimes V_m)
$
is classified by the complex numbers $x_1$, \dots, $x_n$, $h
\in
\C^*$ introduced in Lemma~\ref{lem:LocalRepInvariants}.
Namely, two such local representations are isomorphic as
local representations if and only if they are associated to
the same edge weights $x_i \in \C^*$ and to the same
central load $h$. 

Conversely, any set of edge weights $x_i \in \C^*$ and any $N$--th
root $h$ of $x_1\dots x_n$ can be realized by a local representation
$\rho$. 
\end{prop}

\begin{proof}
This immediately follows from definitions and from
the classification of representations of the
triangle algebra given in Lemma~\ref{lem:TriangleRep}.
\end{proof}

Since Lemma~\ref{lem:LocalRepInvariants} shows how to read the invariants $x_i$ and $h\in \C^*$ from $\rho$, an immediate corollary of Proposition~\ref{prop:ClassifyLocalRep} is the following. 
\begin{cor}
\label{cor:LocalIsom}
Two local representations 
$
\rho\col \T_\lambda^q \to  \End (V_1 \otimes \dots
\otimes V_m)
$
and $
\rho'\col \T_\lambda^q \to  \End (V_1' \otimes \dots
\otimes V_m')
$ 
are isomorphic as local representations if and only if they are isomorphic as representations. \qed
\end{cor}

Recall that $
\rho\col \T_\lambda^q \to  \End (V_1 \otimes \dots
\otimes V_m)
$
and $
\rho'\col \T_\lambda^q \to  \End (V_1' \otimes \dots
\otimes V_m')
$ 
are isomorphic as representations if there exists a linear isomorphism $L\col V_1 '\otimes \dots
\otimes V_m '\to V_1 \otimes \dots
\otimes V_m$ such that 
$
\rho'(P) = L^{-1} \circ \,\rho(P) \circ L
$
for every $P\in \T_\lambda^q$. Corollary~\ref{cor:LocalIsom} says that $L$ can always be chosen to be \emph{tensor-split}, namely to be decomposable 
as a tensor product $L=L_1\otimes \dots \otimes L_m$ of linear isomorphisms $L_j\col V_j \to V_j'$.

\subsection{Fusing local representations}
\label{subsect:Fusing}

Let $\lambda$ be an ideal triangulation of the punctured
surface $S$. Let $R$ be the (possibly disconnected) surface
obtained by splitting
$S$ open along certain edges $\lambda_{i_1}$,
$\lambda_{i_2}$, \dots,
$\lambda_{i_l}$ of
$\lambda$.  Let $\mu$ be
the triangulation of $R$ induced by $\lambda$. 

Conversely, we can think that $S$, with its ideal
triangulation $\lambda$, is obtained from $R$ triangulated
by $\mu$ by gluing together disjoint pairs of edges of
$\mu$ which are contained in the boundary $\partial
R$. 

Note that $\lambda$ and $\mu$ have the same faces
$T_1$, $T_2$, \dots, $T_m$. As a consequence, the
Chekhov-Fock algebras $\T^q_\lambda$ and $\mathcal
T^q_\mu$ are both subalgebras of the tensor product
$\bigotimes_{j=1}^m \T_{T_j}^q$ of the triangle
algebras associated to the faces $T_j$. In addition, $\T^q_\lambda$
is contained in $\T^q_\mu$. 

From this description, we immediately see that a local
representation 
$
\sigma\col \T_\mu^q \to  \End (V_1 \otimes
\dots
\otimes V_m)
$
restricts to a representation
$
\rho\col \T_\lambda^q \to  \End (V_1 \otimes \dots
\otimes V_m)
$
which is also local. We will say that $\rho$ is obtained by
\emph{fusing} the local representation $\sigma$.

For instance, if we split $S$ along all the edges of
$\lambda$, the split surface $R$ is a disjoint union of
triangles, and local representations of $\mathcal
T^q_\lambda$ coincide with representations obtained by
fusing together irreducible representations of the corresponding
triangle algebras. 

\section{The quantum Teichm\"uller space}
\label{sect:QTS}

\subsection{Switching from one ideal triangulation to another}
\label{subsect:Switch}
We now investigate what happens as we switch from one ideal triangulation $\lambda$ to another one  $\lambda'$, with respective Chekhov-Fock
algebras $\T_\lambda^q = \C [X_1^{\pm1},
\dots, X_n^{\pm1}]_\lambda^q$ and
$\T_{\lambda'}^q=\C [X_1'{}^{\pm1}, \dots,
X_n'{}^{\pm1}]_{\lambda'}^q$.
These
skew-polynomial algebras satisfy the so-called Ore
condition, and consequently admit fraction
division algebras 
$\widehat{\T}_\lambda^q = \C (X_1,
\dots, X_n)_\lambda^q$ and
$\widehat{\T}_{\lambda'}=\C (X_1',
\dots, X_n'{})_{\lambda'}^q$; see for instance
\cite{Coh, Kas}. In practice,
$\widehat{\T}_\lambda^q = \C (X_1,
\dots, X_n)_\lambda^q$ consists of
non-commutative rational fractions in the
variables $X_1$\dots $X_n$ which are manipulated
according to the skew-com\-mutativity relations
$X_iX_j = q^{\sigma_{ij}} X_jX_i$. 

Chekhov and Fock introduce \emph{coordinate change isomorphisms}
$ \Phi_{\lambda\lambda'}^q \col \widehat{\T}_{\lambda'}^q
 \to \widehat{\T}_\lambda^q
$ which satisfy the following conditions:
\begin{enumerate}
\item (Composition Relation) $\Phi_{\lambda\lambda''}^q =\Phi_{\lambda\lambda'}^q \circ \Phi_{\lambda'\lambda''}^q $ for every $\lambda$, $\lambda'$ and $\lambda''$;

\item (Naturality Relation) if the diffeomorphism $\phi \col S \to R$ respectively sends the ideal triangulations $\lambda$, $\lambda'$ of the surface $S$ to the ideal triangulations $\mu$, $\mu'$ of the surface $R$, then the natural identifications $\T_\lambda ^q \cong \T_\mu^q$ and $\T_{\lambda'}^q \cong \T_{\mu'}^q$ induced by $\phi$ identify $\Phi_{\lambda\lambda'}^q$ to $\Phi_{\mu'\mu'}^q$;

\item (Fusion Relation) if $S$ is obtained by fusing another surface $R$ along certain components of $\partial R$, and if $\lambda$ and $\lambda'$ respectively come from ideal triangulations $\mu$ and $\mu'$ of $R$, then 
$ \Phi_{\lambda\lambda'}^q \col \widehat{\T}_{\lambda'}^q
 \to \widehat{\T}_\lambda^q
$ 
is the restriction of $ \Phi_{\mu\mu'}^q \col \widehat{\T}_{\mu'}^q
 \to \widehat{\T}_{\mu}^q
$ 
to $\widehat{\T}_{\lambda'}^q \subset \widehat{\T}_{\mu'}^q$. 
\end{enumerate}

Explicit formulas for the coordinate change isomorphisms $\Phi_{\lambda\lambda'}^q$ can be found in \cite{CF, Liu1, BonLiu}. A fundamental example is that where $S$ is a square, namely a disk with 4 punctures on its boundary, and where the edges of $\lambda$ and $\lambda'$ are indexed as shown on Figure~\ref{fig:DiagEx}. Then,
\begin{align*}
\Phi_{\lambda\lambda'}^q(X_1') &=  X_1^{-1},&&\\
\Phi_{\lambda\lambda'}^q(X_2')  &= (1+qX_1) X_2,
&\Phi_{\lambda\lambda'}^q(X_3') 
& = (1+qX_1^{-1})^{-1} X_3,\\ 
\Phi_{\lambda\lambda'}^q(X_4')  &= (1+qX_1) X_4 
&\Phi_{\lambda\lambda'}^q(X_5')
& = (1+qX_1^{-1})^{-1} X_5.
\end{align*}

\begin{figure}[htb]
\SetLabels
( .13 * .5) $\lambda_1$ \\
( .17 *  1.01) $\lambda_2$ \\
( .34 * .5 ) $\lambda_3$ \\
( .17 * -.08 )  $\lambda_4$ \\
( -.02 * .5 ) $\lambda_5$ \\
(  .87* .5 ) $\lambda_1'$ \\
( .83 * 1.01 ) $\lambda_2'$ \\
( 1.02 * .5 ) $\lambda_3'$ \\
( .83 * -.08 ) $\lambda_4'$ \\
( .66 * .5 ) $\lambda_5'$ \\
(  .1*  .7)  $T_1$\\
( .21 * .3 )  $T_2$\\
( .9 * .7 )  $T_1'$\\
( .79 * .3 )  $T_2'$\\
\endSetLabels
\centerline{
\AffixLabels{\includegraphics{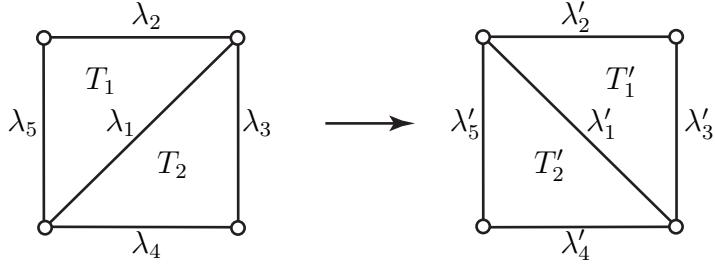}}}
\caption{A diagonal exchange}
\label{fig:DiagEx}
\end{figure}

The fusion point of view can actually be used to somewhat simplify the case-by-case analysis required by the analysis of non-embedded diagonal exchanges in \cite{Liu1, BonLiu}.

In \cite{Bai1}, it is proved that the coordinate change isomorphisms $\Phi_{\lambda\lambda'}^q$ are completely determined by the Composition, Naturality and Fusion Conditions, modulo some uniform renormalization of the $X_i$. The combination of the Naturality and Fusion Conditions is equivalent to the Locality Condition considered in \cite{Bai1}. 

The \emph{quantum Teichm\"uller space} $\T_S^q$ of the punctured surface $S$ is defined as the quotient
$$
\T_S^q = \bigsqcup_\lambda \widehat{\T}_\lambda^q /\sim \   = \bigsqcup_\lambda\C (X_1,
\dots, X_n)_\lambda^q  /\sim
$$
of the disjoint union of the $ \widehat{\T}_\lambda^q$ of all ideal triangulations $\lambda$ of $S$, where the equivalence relation $\sim$ identifies $\widehat{\T}_\lambda^q = \C (X_1,
\dots, X_n)_\lambda^q$  to 
$\widehat{\T}_{\lambda'}^q=\C (X_1',
\dots, X_n'{})_{\lambda'}^q$ by the coordinate change isomorphism $\Phi_{\lambda\lambda'}^q$. Because the $\Phi_{\lambda\lambda'}^q$ are algebra isomorphisms, the quantum Teichm\"uller space  $\T_S^q$ inherits an algebra structure from the $\widehat{\T}_\lambda^q $.

\subsection{Representations of the quantum Teichm\"uller space}
By the above definitions, a representation of the quantum Teichm\"uller space $\T_S^q$ should be equivalent to the data, for each ideal triangulation $\lambda$, of an algebra homomorphism $\rho_\lambda \col \widehat{\T}_\lambda^q \to \End(V)$ from  $\widehat{\T}_\lambda^q $ to the algebra of endomorphisms of a vector space $V$, in such a way that $\rho_{\lambda'} = \rho_\lambda\circ \Phi_{\lambda\lambda'}^q$ for every $\lambda$, $\lambda'$. 

This definition would be natural but, if we are interested in finite-dimensional representations, leads to several stumbling blocks. 

The main difficulty is that there is no algebra homomorphism $\rho_\lambda \col \widehat{\T}_\lambda^q \to \End(V)$ with $V$ finite-dimensional. Indeed, every rational fraction $R\not=0$ is invertible in $\widehat{\T}_\lambda^q$, so that $\rho_\lambda(R)$ would have to be invertible in $\End(V)$. However, when $V$ is finite-dimensional, so is $\End(V)$ whereas $\widehat{\T}_\lambda^q$ is infinite-dimensional as a vector space. It follows that there are many $R\not= 0$ such that $\rho_\lambda(R)$ is equal to $0$, and in particular is not invertible.

One way to solve this problem is to restrict attention to Laurent polynomials, and to consider algebra homomorphisms $\rho_\lambda\col \T_\lambda^q = \C [X_1^{\pm1},
\dots, X_n^{\pm1}]_\lambda^q \to \End(V)$. Indeed, we observed in Section~\ref{sect:LocalRep} that the Chekhov-Fock algebra $\T_\lambda^q $ does admit many finite-dimensional representations.

Unfortunately, this quick fix creates another problem: Since the coordinate change isomorphisms $\Phi_{\lambda\lambda'}^q$ are valued in the fraction division algebra $\widehat{\T}_\lambda^q$ and not in the Chekhov-Fock algebra ${\T}_\lambda^q$, the composition $\rho_\lambda\circ \Phi_{\lambda\lambda'}^q$ does not automatically make sense, so that the condition $\rho_{\lambda'} = \rho_\lambda\circ \Phi_{\lambda\lambda'}^q$ has to be suitably reinterpreted. 

Given an algebra homomorphism $\rho_\lambda\col \T_\lambda^q \to \End(V)$, we will say that the algebra homomorphism 
$$\rho_\lambda\circ \Phi_{\lambda\lambda'}^q \col   \T_{\lambda'}^q = \C [X_1^{\pm1},
\dots, X_n^{\pm1}]_{\lambda'}^q \to \End(V) $$
 \emph{makes sense} if, for every Laurent polynomial $X'\in \T_{\lambda'}^q$, the rational fraction $\Phi_{\lambda\lambda'}^q(X') \in \widetilde{\T}_\lambda^q$ can be written as quotients
$$
\Phi_{\lambda\lambda'}^q(X') = PQ^{-1} = Q'{}^{-1}P'
$$
of Laurent polynomials $P$, $Q$, $P'$, $Q'\in \T_\lambda^q$ such that $\rho_\lambda(Q)$ and $\rho_\lambda(Q')$ are invertible in $\End(V)$. We then define 
$$\rho_\lambda\circ \Phi_{\lambda\lambda'}^q(X')
=\rho_\lambda(P) \rho_\lambda(Q)^{-1} = \rho_\lambda(Q')^{-1}  \rho_\lambda(P') \in \End(V).$$
 One easily sees that $\rho_\lambda\circ \Phi_{\lambda\lambda'}^q(X)$  does not depend on the decomposition of $\Phi_{\lambda\lambda'}^q(X)$ as a quotient of polynomials, and that this defines an algebra homomorphism
$\rho_\lambda\circ \Phi_{\lambda\lambda'}^q  \col \T_{\lambda'}^q  \to \End(V)$.
These two properties strongly use the 2--sided decomposition of each $\Phi_{\lambda\lambda'}^q(X)$.

\begin{defi}
A \emph{representation} of the quantum Teichm\"uller space $\T_S^q$ of the surface $S$ over the finite-dimensional vector space $V$ consists  of the data  of an algebra homomorphism $\rho_\lambda\col \T_\lambda^q = \C [X_1^{\pm1},
\dots, X_n^{\pm1}]_\lambda^q \to \End(V)$ for every ideal triangulation $\lambda$ in such a way that, for every $\lambda$, $\lambda'$, the representation
 $\rho_\lambda\circ \Phi_{\lambda\lambda'}^q  \col \T_{\lambda'}^q  \to \End(V)$ makes sense and is equal to  $\rho_{\lambda'}$.
\end{defi}

Strictly speaking,  a representation of the quantum Teichm\"uller space is not a representation of the algebra $\T_S^q$ in the usual sense. However, this terminology appears more convenient than that of \cite{BonLiu}, where the same object was called  a representation of the polynomial core of the quantum Teichm\"uller space. 

Given two specific ideal triangulations $\lambda$ and $\lambda'$ and a representation $\rho_\lambda\col \T_\lambda^q  \to \End(V_\lambda)$, it is in general quite difficult to check that the representation $\rho_\lambda \circ \Phi_{\lambda\lambda'}^q$ makes sense. This is much easier if we consider all ideal triangulations at the same time.

\begin{lem}
\label{lem:EnuffCheckDiagEx}
Let an algebra homomorphism $\rho_\lambda\col \T_\lambda^q \to \End(V)$ be given for every ideal triangulation $\lambda$. 
Suppose that $\rho_\lambda\circ \Phi_{\lambda\lambda'}^q  \col \T_{\lambda'}^q  \to \End(V)$ is equal to $\rho_{\lambda'}$ whenever  $\lambda$ and $\lambda'$  differ by a diagonal exchange.  Then the $\rho_\lambda$ form a representation of the quantum Teichm\"uller space of $S$. 
\end{lem}
\begin{proof}
This is the main result of \cite[\S6]{BonLiu}. The proof is somewat more difficult that one could have anticipated at first glance, as one needs to avoid invoking the inverse of an endomorphism which is not invertible. In particular, the argument used in \cite{BonLiu} depends on the specific formulas for the $\Phi_{\lambda\lambda'}^q$ associated to diagonal exchanges. 
\end{proof}

\section{Local representations of the quantum Teichm\"uller space}

A \emph{local representation} of the quantum Teichm\"uller space $\T_S^q$ is a representation $\{\mu_\lambda : \T_\lambda^q \to \End(V) \}_{\lambda \in \Lambda(S)}$ of $\T_S^q$ such that each $\mu_\lambda$ is  isomorphic to a local representation $\rho_\lambda : \T_\lambda^q \to \End(V_\lambda)$. 

In practice, it is more convenient to put emphasis on the local representations $\rho_\lambda$. Namely, up to isomorphism, a local representation of $\T_S^q$ consists of the data of a local representation 
$\rho_\lambda \col \T_\lambda^q \to \End(V_\lambda)$ for every ideal triangulation $\lambda$, in such a way that $\rho_\lambda\circ \Phi_{\lambda\lambda'}^q   \col \T_{\lambda'}^q \to \End(V_\lambda)$ is isomorphic  to  $\rho_{\lambda'} \col \T_{\lambda'}^q \to \End(V_{\lambda'}) $ for every $\lambda$, $\lambda'$. 
Recall from the definition of local representations that each vector space $V_\lambda$ splits as a tensor product $V_\lambda = V_1 \otimes V_2 \otimes \dots \otimes V_m$ where each factor $V_j$ is an $N$--dimensional vector space attached to the $j$--th face of $\lambda$. Recall also that the local representation $\rho_\lambda$ consists of more than the algebra homomorphism $\rho_\lambda \col \T_\lambda^q \to \End(V_\lambda)$, but also includes the data of representations $\rho_j\col  \col \T_{T_j}^q \to \End(V_j)$ well-defined up to a certain equivalence relation; however, this distinction matters only when an edge of $\lambda$ corresponds to two sides of the same face. 

By Proposition~\ref{prop:ClassifyLocalRep}, the local representation $\rho_\lambda : \T_\lambda^q \to \End(V_\lambda)$ is classified by the data of a complex weight $x_i \in \C^*$ attached to each edge $\lambda_i$ of $\lambda$, plus a central load $h = \sqrt[N]{x_1x_2\dots x_n}$. We want to determine how these edge weights and central load change as we switch from $\lambda$ to another ideal triangulation $\lambda'$. 

It turns out that the corresponding change of edge weights is related to the  coordinate change isomorphism $\Phi_{\lambda\lambda'}^1 \col \T_{\lambda'}^1 \to \T_\lambda^1 $ associated to the non-quantum case where $q=1$. Indeed, in this non-quantum (more traditionally called semi-classical) case where $q=1$, the Chekhov-Fock algebras $\T_\lambda^1$ and $\T_{\lambda'}^1$ are both isomorphic to the algebra of commutative rational fractions $\C(X_1, X_2, \dots, X_n)$. Then, the coordinate change isomorphism $$
\Phi_{\lambda\lambda'}^1 \col \T_{\lambda'}^1 = \C(X_1', X_2', \dots , X_n') \to \C(X_1, X_2, \dots, X_n) = \T_\lambda^1
$$ 
determines a  rational function $\phi_{\lambda\lambda'}\col \C^N \to \C^N$ (only partially defined as a function) by the property that the $i$--th coordinate of $\phi_{\lambda\lambda'}(X_1, X_2, \dots, X_n)$ is equal to $\Phi_{\lambda\lambda'}^1(X_i')$. Conversely, $\phi_{\lambda\lambda'}$ determines $\Phi_{\lambda\lambda'}^1$ because $\Phi_{\lambda\lambda'}^1(F) = F\circ \phi_{\lambda\lambda'}$ for every $F \in \C(X_1', X_2', \dots, X_n')$. 

Our hypothesis that $q$ be a primitive $N$--th root of $(-1)^{N+1}$ is critical for the following result. 

\begin{prop}
\label{prop:RepGivesWeights}
Let 
$\rho= \{\rho_\lambda : \T_\lambda^q \to \End(V_\lambda) \}_{\lambda \in \Lambda(S)}$ be a local representation of the quantum Teichm\"uller space $\T_S^q$. If $\rho_\lambda$ is classified by the edge weights $x_i\in \C^*$ and by the central load $h\in \C^*$, and if $\rho_{\lambda'}$ is classified by the edge weights $x_{i'}'\in \C^*$ and by the central load $h'$, then $h=h'$ and 
$$
(x_1', x_2', \dots, x_n') = \phi_{\lambda\lambda'}(x_1, x_2, \dots, x_n).
$$
\end{prop}

In particular, the edge weights $(x_1, x_2, \dots, x_n)\in \C^n$ are in the domain of $\phi_{\lambda\lambda'}$, so that $ \phi_{\lambda\lambda'}(x_1, x_2, \dots, x_n)$ is well-defined.

\begin{proof}[Proof of Proposition~\ref{prop:RepGivesWeights} when $S$ is a square]
The square $S$ admits only the two ideal triangulations represented in  Figure~\ref{fig:DiagEx}. 

Recall that the weights $x_i$ and $x_i'$ are defined by the property that $\rho_\lambda(X_i^N) = x_i \Id_{V_\lambda}$ and $\rho_{\lambda'}(X_i'{}^N) = x_i' \Id_{V_{\lambda'}}$. 
The argument is then a simple algebraic computation using the explicit formulas for $\Phi_{\lambda\lambda'}^q$ in the case of the square,  as in  \cite[Lem.~27]{BonLiu} and in the end of \cite{CF}. 

For instance, with the edge indexing of  Figure~\ref{fig:DiagEx},
\begin{align*}
\rho_\lambda\circ \Phi_{\lambda\lambda'}^q(X_2'{}^N )
&=\bigl( \rho_\lambda(X_2) + q\rho_{\lambda}(X_1)\rho_\lambda(X_2)\bigr )^N \\
&=  \rho_\lambda(X_2)^N + \bigl(q\rho_{\lambda}(X_1)\rho_\lambda(X_2)\bigr )^N \\
&=  \rho_\lambda(X_2)^N + \rho_{\lambda}(X_1)^N \rho_\lambda(X_2)^N \\
& = x_2 \Id_{V_\lambda} + x_1x_2 \Id_{V_\lambda}
\end{align*}
where the second equality combines the Quantum Binomial Formula with the hypothesis that $q^2$ is a primitive $N$--th root of unity, while the third one uses the condition that $q^N=(-1)^{N+1}$. 
Since $\rho_\lambda\circ \Phi_{\lambda\lambda'}^q(X_2'{}^N )$ is conjugate to $\rho_{\lambda'}(X_2'{}^N)= x_2' \Id_{V_{\lambda'}}$, it follows that $x_2' = (1+x_1)x_2$. 

Similar computations give that $x_1' = x_1^{-1}$, $x_3'=(1+x_1^{-1})x_3$, $x_4' = (1+x_1)x_4$ and $x_5'=(1+x_1^{-1})^{-1}x_5$. Namely, by the explicit formula for the non-quantum coordinate change isomorphism $\Phi_{\lambda\lambda'}^1$, the edge weight system $(x_1', x_2', x_3', x_4', x_5')$ is equal to $\phi_{\lambda\lambda'} (x_1, x_2, x_3, x_4, x_5)$. 
\end{proof}

\begin{proof}[Proof of Proposition~\ref{prop:RepGivesWeights} in the general case]
By \cite{Har, Pen, Hat}, any two ideal triangulations can be connected by a sequence of  diagonal exchanges. Since the $\phi_{\lambda\lambda'}$ satisfy the Composition Relation of Section~\ref{subsect:Switch} (with the order reversed), it suffices to check Proposition~\ref{prop:RepGivesWeights} when $\lambda$ and $\lambda'$ differ only by a diagonal exchange. 

If we split open the surface $S$ along the sides of the square where the diagonal exchange takes place, we obtain a square $Q$ and a (possibly disconnected) surface $R$. The ideal triangulation $\lambda$ induces an ideal triangulation $\lambda_Q$ of $Q$ and an ideal triangulation $\lambda_R$ of $R$. Similarly, $\lambda'$ defines ideal triangulations $\lambda_Q'$ and $\lambda_R'$ of $Q$ and $R$. The two ideal triangulations $\lambda_R$ and $\lambda_R'$ coincide, while $\lambda_Q$ and $\lambda_Q'$ differ only by the choice of a diagonal for the square $Q$. 

By definition of local representations, $\rho_\lambda \col \T_\lambda^q \to \End(V_\lambda)$ is obtained by fusing together local representations $\rho_{\lambda_Q }\col \T_{\lambda_Q}^q \to \End(V_{\lambda_Q})$ and $\rho_{\lambda_R} \col \T_{\lambda_R}^q \to \End(V_{\lambda_R})$. By the Fusion Relation of Section~\ref{subsect:Switch}, it follows that $\rho_\lambda \circ \Phi_{\lambda\lambda'}^q$ is obtained by fusing together the representations $\rho_{\lambda_Q} \circ \Phi_{\lambda_Q\lambda_Q'}^q$ and $\rho_{\lambda_R} \circ \Phi_{\lambda_R\lambda_R'}^q = \rho_{\lambda_R}$ (since $\Phi_{\lambda_R\lambda_R'}^q$ is just the identity of $\T_{\lambda_R}^q = \T_{\lambda_R'}^q$). 
The property then immediately follows from the case of the square, which we just  analyzed, as well as from the fact that $\phi_{\lambda\lambda'}$ is obtained by fusing together $\phi_{\lambda_Q\lambda_Q'}$ and $\phi_{\lambda_R\lambda_R'}=\Id_{V_{\lambda_R}}$. This second property is a consequence of the fact that the non-quantum coordinate change isomorphisms $\Phi_{\lambda\lambda'}^1$ satisfy the Fusion Relation. 
\end{proof}

There is a converse to Proposition~\ref{prop:RepGivesWeights}. 

\begin{prop}
\label{prop:WeightsGiveRep}
Suppose that we are given a system of edge weights $x_i\in \C^*$  for every ideal triangulation $\lambda$ in such a way that, for any two ideal triangulations $\lambda$ and $\lambda'$,  the edge weights $(x_1, x_2, \dots, x_n)$ and $(x_1', x_2', \dots, x_n')$ respectively associated to $\lambda$ and $\lambda'$ are related by the property that:
$$
(x_1', x_2', \dots, x_n') = \phi_{\lambda\lambda'}(x_1, x_2, \dots, x_n).
$$
(In particular, $(x_1, x_2, \dots, x_n)$ is in the domain of $\phi_{\lambda\lambda'}$.)
In addition, let $h\in \C^*$ be an $N$--th root for the product $x_1x_2\dots x_n$ of the edge weights associated to an arbitrary ideal triangulation $\lambda_0$. 

Then, there exists a local representation $\rho= \{\rho_\lambda : \T_\lambda^q \to \End(V_\lambda) \}_{\lambda \in \Lambda(S)}$  of the quantum Teichm\"uller space $\T_S^q$ such that, for every ideal triangulation $\lambda$, the edge weights and central load classifying the local representation $\rho_\lambda$ as in Proposition~\ref{prop:ClassifyLocalRep} are exactly the given edge weights  $x_i$ and $N$--th root $h$. In addition, $\rho$ is unique up to isomorphism of local representations. 
\end{prop}

\begin{proof}
As a preliminary remark we note that, for any two ideal triangulations $\lambda$ and $\lambda'$, the products $x_1x_2 \dots x_n$ and $x_1'x_2'\dots x_n'$ of their respective edge weights are equal. This can easily seen by inspection of the formulas  for $\phi_{\lambda\lambda'}$ (see also \cite[Prop.~14]{Liu1} and Lemma~\ref{lem:PeriphLoad} below). Therefore, $h$ is an $N$--th root for the product $x_1x_2 \dots x_n$ of the edge weights associated to every ideal triangulation $\lambda$. 

For every ideal triangulation $\lambda$, Proposition~\ref{prop:ClassifyLocalRep} associates a local representation $\rho_\lambda : \T_\lambda^q \to \End(V_\lambda)$ to the edge weights $x_1$, $x_2$, \dots, $x_n$ and to the $N$--th root $h=\sqrt[N]{x_1x_2\dots x_n}$.

To prove that the $\rho_\lambda$ form a local representation of $\T_S^q$, Lemma~\ref{lem:EnuffCheckDiagEx} says that it suffices to show that $\rho_\lambda \circ \Phi_{\lambda\lambda'}$ makes sense and is isomorphic to $\rho_{\lambda'}$ whenever the ideal triangulations $\lambda$ and $\lambda'$ differ only by a diagonal exchange. As in the proof of Proposition~\ref{prop:RepGivesWeights},   split the surface $S$ open along the sides of the square where the diagonal exchange takes place to obtain a square $Q$ and a possibly disconnected surface $R$. The ideal triangulation $\lambda$ and $\lambda'$ respectively induce ideal triangulation $\lambda_Q$ and $\lambda_Q'$ of $Q$, and ideal triangulations $\lambda_R=\lambda_R'$ of $R$. By definition of local representations, $\rho_\lambda$ is obtained by fusing together a local representation $\rho_{\lambda_Q}$ of $\T_{\lambda_Q}^q$ and a local representation $\rho_{\lambda_R}$ of $\T_{\lambda_R}^q$.

Although everything could be done ``by hand'' (see for instance our explicit computations in \cite{BaiBon}), we will take advantage of the classification of irreducible representations of $\T_Q^q$ in \cite{BonLiu}. Theorems~21 and 22 of \cite{BonLiu} (straightforwardly extended to surfaces with boundary) and  Proposition~\ref{prop:ClassifyLocalRep}  of this paper show that irreducible and local representations of $\T_{\lambda_Q}^q$ have the same dimension $N^2$, and are classified by the same edge weights and the same central load. It follows that every $N^2$--dimensional representation of $\T_{\lambda_Q}^q$ is irreducible, and is isomorphic to a local representation. 

Because the weights $
(x_1', x_2', \dots, x_n') = \phi_{\lambda\lambda'}(x_1, x_2, \dots, x_n)
$
are finite and non-zero, \cite[Lemma~29]{BonLiu} shows that the representation $\rho_{\lambda_Q}\circ \Phi_{\lambda_Q\lambda_Q'}^q$ of $\T_{\lambda_Q}^q$ makes sense. It follows that the representation $\rho_{\lambda}\circ \Phi_{\lambda\lambda'}^q$ of $\T_{\lambda}^q$, obtained by fusing $\rho_{\lambda_Q}\circ \Phi_{\lambda_Q\lambda_Q'}^q$ and $\rho_{\lambda_R}\circ \Phi_{\lambda_R\lambda_R'}^q = \rho_{\lambda_R}$ together, makes sense. By the above observation, $\rho_{\lambda_Q}\circ \Phi_{\lambda_Q\lambda_Q'}^q$ is isomorphic to a local representation. Therefore, 
$\rho_{\lambda}\circ \Phi_{\lambda\lambda'}^q$ is isomorphic to a local representation. By Proposition~\ref{prop:RepGivesWeights},  $\rho_{\lambda}\circ \Phi_{\lambda\lambda'}^q$ is classified by $
\phi_{\lambda\lambda'}(x_1, x_2, \dots, x_n)= (x_1', x_2', \dots, x_n')  
$ and the central load $h$. Proposition~\ref{prop:ClassifyLocalRep} then shows that $\rho_{\lambda}\circ \Phi_{\lambda\lambda'}^q$ is isomorphic to $\rho_{\lambda'}$. 

Having proved that $\rho_{\lambda}\circ \Phi_{\lambda\lambda'}^q$ makes sense and is isomorphic to $\rho_{\lambda'}$ whenever $\lambda$ and $\lambda'$ differ by a diagonal exchange,  Lemma~\ref{lem:EnuffCheckDiagEx} shows that the $\rho_\lambda : \T_\lambda^q \to \End(V_\lambda)$ form a local representation $\rho$ of the quantum Teichm\"uller space $\T_S^q$.

The uniqueness of $\rho$ up to isomorphism is an immediate consequence of Proposition~\ref{prop:ClassifyLocalRep}.
\end{proof}

Propositions~\ref{prop:RepGivesWeights} and \ref{prop:WeightsGiveRep} can be combined in the following statement. 

\begin{thm}
\label{thm:LocalAndWeights}
Let $q$ be a primitive $N$--th root of $(-1)^{N+1}$. 
There is a one-to-one correspondence between:
\begin{enumerate}
\item isomorphism classes of local representations $\rho= \{\rho_\lambda : \T_\lambda^q \to \End(V_\lambda) \}_{\lambda \in \Lambda(S)}$;
\item data associating to each ideal triangulation $\lambda$  a system of edge weights $x_i\in \C^*$ and a choice of $N$--th root $h = \sqrt[N]{x_1x_2\dots x_n}$ in such a way that, for any two ideal triangulations $\lambda$ and $\lambda'$,  the edge weights $(x_1, x_2, \dots, x_n)$ and $(x_1', x_2', \dots, x_n')$ and the roots $h = \sqrt[N]{x_1x_2\dots x_n}$ and $h' = \sqrt[N]{x_1'x_2'\dots x_n'}$ respectively associated to $\lambda$ and $\lambda'$ are such that
$$
(x_1', x_2', \dots, x_n') = \phi_{\lambda\lambda'}(x_1, x_2, \dots, x_n).
$$
and $h'=h$. \qed
\end{enumerate}
\end{thm}

We will now reinterpret Theorem~\ref{thm:LocalAndWeights} in a more geometric setting.

\section{Pleated surfaces}
\label{sect:Pleated}

There is a geometric situation where complex edge weights
for an ideal triangulation of a surface also occur, namely
the exponential shear-bend parameters of a pleated surface
with this ideal triangulation as pleating locus. 

In this section, we restrict attention to the case where the
boundary of $S$ is empty, namely where $S$ is obtained
by removing $p\geq 1$ punctures from a closed oriented surface
$\bar S$. 

 Let $\widetilde S$ be the universal
covering of $S$, let $\widetilde\lambda\subset \widetilde S$ be
the preimage of the edges of $\lambda$, and let $ \Isom (\HH^3)$
denote the group of orientation-preserving isometries of the hyperbolic 3--space $\HH^3$. A
\emph{pleated surface} with pleated along $\lambda$ is 
a pair $\bigl(
\widetilde f , r\bigr)$ consisting of  a continuous map
$\widetilde f \col \widetilde S \to \HH^3$ and of a group homomorphism
$r\col \pi_1(S) \to \Isom (\HH^3)$ such that:
\begin{enumerate}

\item $\widetilde f$
homeomorphically sends each component of
$\widetilde \lambda$
to a complete geodesic of $\mathbb
H^3$;

\item $\widetilde f$
homeomorphically sends the
closure of each component of
$\widetilde S-
\widetilde \lambda$ to an ideal
triangle in $\mathbb H^3$, namely
one whose three vertices are on
the sphere at infinity
$\partial_\infty \mathbb H^3$ of
$\mathbb H^3$;

\item $\widetilde f$ is
$r$--equivariant in the sense
that $\widetilde f(\gamma
\widetilde x) = r(\gamma)
\widetilde f(
\widetilde x)$ for every
$\widetilde x \in \widetilde S$
and $\gamma \in \pi_1(S)$. 
\end{enumerate}

If $\lambda_i$ is an edge of $\lambda$, separating two (possibly
equal) faces
$T_j$ and
$T_k$, lift $\lambda_i$ to a component $\widetilde\lambda_i$ of
$\widetilde\lambda$ separating two components $\widetilde T_j$ and
$\widetilde T_k$ of $\widetilde S - \widetilde \lambda$ respectively
lifting the interiors of $T_j$ and $T_k$. 

The closure of the two
ideal triangles $\widetilde f \bigl( \widetilde T_j \bigr)$ and 
$\widetilde f \bigl( \widetilde T_k \bigr)$ is an oriented wedge bent by
an angle of $\theta \in \R/2\pi\Z$; here we choose $\theta$ to be the external
dihedral angle of this wedge
along $\widetilde \lambda_i$, namely $\pi$ minus the internal dihedral angle, so that $\theta=\pi$ when the two ideal triangles $\widetilde f \bigl( \widetilde T_j \bigr)$ and 
$\widetilde f \bigl( \widetilde T_k \bigr)$ coincide. In addition, consider on
$\widetilde\lambda_i$ the respective projections $w_j$ and $w_k$ of the
vertices of the ideal triangles   $\widetilde f \bigl( \widetilde T_j
\bigr)$  and 
$\widetilde f \bigl( \widetilde T_k \bigr)$ that are not in
$\widetilde\lambda_i$. If we orient $\widetilde\lambda_i$ as part of the
boundary of the oriented triangle $\widetilde f \bigl( \widetilde T_j
\bigr)$, let the \emph{shear parameter} $t$ be the oriented distance from
$w_j$ to $w_k$ in  $\widetilde\lambda_i$. 

By definition, the complex number $x_i = \mathrm e^{t+ \mathrm i \theta}$
is the
\emph{exponential shear-bend parameter} of the pleated surface along the
edge $\lambda_i$. It is immediate that $x_i$ is independent of the
choices made in the definition. 

Note that, for a pleated surface 
$\bigl( \widetilde f , r\bigr)$,
the homomorphism $r\col \pi_1(S)
\rightarrow
\Isom(\HH^3)$ is
completely determined by the map 
$\widetilde f \col
\widetilde S \rightarrow \mathbb
H^3$. The map $\widetilde f$ adds
more data to $r$ as follows. Let
$A \subset S$ be the union of
small annulus neighborhoods of
all the punctures of $S$.  There
is a one-to-one correspondence
between the components of the
preimage $\widetilde A$ of
$A$ in $\widetilde S$ and the \emph{peripheral
subgroups} of
$\pi_1(S)$, namely the image subgroups
of the homomorphisms $\pi_1(A)
\rightarrow \pi_1(S)$ defined by
all possible choices of base
points and paths joining these
base points. For a component
$\widetilde A_\pi$ of
$\widetilde A$ corresponding to a
peripheral subgroup $\pi \subset
\pi_1(S)$, the images under
$\widetilde f$ of the triangles of
$\widetilde S - \widetilde
\lambda$ that meet
$\widetilde A_\pi$ all have a
vertex $\xi_\pi$ in common in
$\widehat {\mathbb C} =
\partial_\infty \mathbb H^3$, and
this vertex is fixed by $r(\pi)$.
Therefore, $\widetilde f$
associates to each peripheral
subgroup
$\pi$ of $\pi_1(S)$ a point
$\xi_\pi \in
\partial_\infty \mathbb H^3$
which is fixed under $r(\pi)$. In
addition this assignment is
\emph{$r$--equivariant} in the sense that
$\xi_{\gamma\pi\gamma^{-1}}=
r(\gamma) \xi_\pi$ for every
$\gamma \in \pi_1(S)$. 

By definition, an \emph{enhanced
homomorphism} $(r,
\{\xi_\pi\}_{\pi \in
\Pi})$ from
$\pi_1(S)$ to
$\Isom(\HH^3)$
consists of a group homomorphism
$r\col \pi_1(S) \rightarrow
\Isom(\HH^3)$
together with an
$r$--equivariant assignment of a
fixed point
$\xi_\pi \in \partial_\infty
\mathbb H^3$ to each peripheral
subgroup $\pi$ of $\pi_1(S)$.
Here $\Pi$ denotes the set of
peripheral subgroups of
$\pi_1(S)$. By abuse of notation,
we will often write $r$ instead
of 
$(r, \{\xi_\pi\}_{\pi \in
\Pi})$ of
$\pi_1(S)$.

Two such enhanced homomorphisms $(r,
\{\xi_\pi\}_{\pi \in
\Pi})$ and $(r',
\{\xi_\pi'\}_{\pi \in
\Pi})$ are \emph{conjugate} by $A\in \Isom(\HH^3)$
if $r'(\gamma) = A r(\gamma) A^{-1}$ for every
$\gamma\in \pi_1(S)$ and if $\xi_\pi' = A(\xi_\pi)$ for
every peripheral subgroup $\pi\in \Pi$. 

An enhanced group homomorphism $(r,
\{\xi_\pi\}_{\pi \in
\Pi})$ of
$\pi_1(S)$ \emph{realizes} the ideal triangulation $\lambda$ if it is associated to a pleated surface $\bigl(
\widetilde f , r\bigr)$ pleated along $\lambda$. It is \emph{peripherally generic} if it realizes every ideal triangulation $\lambda$ of $S$. The terminology is justified by the fact that this property is essentially equivalent to (although slightly weaker than) the fact that distinct peripheral subgroups $\pi\in \Pi$ are associated to distinct points $\xi_\pi \in \partial _\infty \HH^3$ in the enhancement; compare the proof of Lemma~\ref{lem:InjectivePeriphGen}.

\begin{prop}
\label{prop:Pleated}
Let $\lambda$ be an ideal triangulation for the surface
$S$, with $\partial S= \emptyset$. The map which, to a pleated surface pleated along $\lambda$, associates its exponential shear-bend coordinates induces a one-to-one correspondence between conjugacy classes of enhanced group homomorphisms $(r,
\{\xi_\pi\}_{\pi \in
\Pi})$ from 
$\pi_1(S)$ to
$\Isom(\HH^3)$ realizing $\lambda$ and systems of non-zero edge
weights $x_i \in \C^*$ for the
edges
$\lambda_i$ of $\lambda$.
\end{prop}

\begin{proof}
Reconstructing a pleated surface from its shear-bend coordinates is completely elementary. See for instance \cite{CEG} or \cite{Bon97}, or compare the end of the proof of Lemma~\ref{lem:InjectivePeriphGen}. 
\end{proof}

In practice, it is not always easy to determine whether a general enhanced group homomorphism from $\pi_1(S)$ to $\Isom(\HH^3)$ realizes a given ideal triangulation. However, the following simple criterion is quite useful, since many (most?) group homomorphisms  $r\col \pi_1(S) \to \Isom(\HH^3)$ that have some geometric significance are injective. 

\begin{lem}
\label{lem:InjectivePeriphGen}
Every enhanced homomorphism $(r,
\{\xi_\pi\}_{\pi \in
\Pi})$ where $r\col \pi_1(S) \to \Isom(\HH^3)$ is injective is peripherally generic, namely realizes every ideal triangulation $\lambda$. 
\end{lem}

\begin{proof} If
$\widetilde\lambda_i \subset \widetilde S$ is the
lift of an edge $\lambda_i$ of $\lambda$, its two
ends each define two parabolic subgroups $\pi$,
$\pi'\in \Pi$. We claim that the two fixed points
$\xi_\pi$, $\xi_{\pi'} \in \partial_\infty \HH^3$ of $r(\pi)$ and
$r(\pi')$ are distinct. Indeed, $\pi$ and
$\pi'$ generate a free subgroup $F$ of rank 2 in
$\pi_1(S)$. On the other hand, the stabilizer of a point
 $x \in \partial_\infty\HH^3$ in $\Isom(\HH^3)$  is solvable, and consequently cannot contain any free subgroup of rank $>1$. Since $r$ is injective, this proves that $r(\pi)$ and $r(\pi')$ cannot fix the same point  $x \in \partial_\infty\HH^3$. 

Once we have this property, constructing a pleated surface $\bigl ( \widetilde f, r \bigr)$ realizing $\lambda$ is immediate. First define $\widetilde f$ on each $\widetilde \lambda_i$ as above, by sending it to the geodesic of $\HH^3$ joining the  two fixed points $\xi_\pi \not = \xi_{\pi'}\in \partial_\infty\HH^3$. Then continuously extend $\widetilde f$ to $\widetilde S$ by sending each component $T$ of $\widetilde S - \widetilde \lambda$ to the unique totally geodesic triangle bounded in $\HH^3$ by the three geodesics so associated to the boundary components of $T$. 
\end{proof}

Note that an injective group homomorphism $r\col \pi_1(S) \to \Isom(\HH^3)$ admits at most $2^p$ enhancements (where $p$ is the number of punctures of $s$), and exactly one in the geometrically important case where $r$ sends each peripheral subgroup of $\pi_1(S)$  to a parabolic subgroup of $\Isom(\HH^3)$. Indeed, an infinite cyclic subgroup of $\Isom(\HH^3)$ fixes exactly one point when it is parabolic, and two points when it is loxodromic or elliptic. 

\begin{prop}
\label{prop:PleatedCoordChang}
Let $(r, \{\xi_\pi\}_{\pi \in \Pi})$ be an enhanced homomorphism from $\pi_1(S)$ to $\Isom(\HH^3)$ which is peripherally generic, namely realizes every ideal triangulation. In particular, for any two ideal triangulations $\lambda$ and $\lambda'$, it associates weights $x_i\in \C^*$ to the edges of $\lambda$, and weights $x_{i'}'\in \C^*$ to the edges of $\lambda'$. Then, 
$$
(x_1', x_2', \dots, x_n') = \phi_{\lambda\lambda'}(x_1, x_2, \dots, x_n).
$$
where $\phi_{\lambda\lambda'}$ is the rational map associated to the non-quantum coordinate change isomorphism $\Phi_{\lambda\lambda'}^1 \col \C(X_1', \dots, X_n') \to \C(X_1, \dots, X_n)$. 
\end{prop}
\begin{proof} This is a simple computation. See for instance \cite{Fok} or \cite[\S2]{Liu1}.
\end{proof}

Combining Propositions~\ref{prop:Pleated} and \ref{prop:PleatedCoordChang} immediately gives:
\begin{thm}
\label{thm:EnhHomClass}
There is a one-to-one correspondence between:
\begin{enumerate}
\item conjugation classes of enhanced homomorphisms $(r, \{\xi_\pi\}_{\pi \in \Pi})$  from $\pi_1(S)$ to $\Isom(\HH^3)$ that are peripherally generic;
\item data associating to each ideal triangulation $\lambda$  a system of edge weights $x_i\in \C^*$  in such a way that, for any two ideal triangulations $\lambda$ and $\lambda'$,  the edge weights $(x_1, x_2, \dots, x_n)$ and $(x_1', x_2', \dots, x_n')$  respectively associated to $\lambda$ and $\lambda'$ are such that
$$
(x_1', x_2', \dots, x_n') = \phi_{\lambda\lambda'}(x_1, x_2, \dots, x_n).
$$
\vskip -\belowdisplayskip
\vskip -\baselineskip
\qed
\end{enumerate}

\end{thm}

Given an enhanced homomorphism  $(r,
\{\xi_\pi\}_{\pi \in
\Pi})$ realizing the ideal triangulation $\lambda$, consider the edge
weights $x_i \in \C^*$  associated to $(r,
\{\xi_\pi\}_{\pi \in
\Pi})$ by Proposition~\ref{prop:Pleated}. Define the
\emph{total peripheral load} of $(r,
\{\xi_\pi\}_{\pi \in
\Pi})$ as the product
$$h_r = x_1\dots x_n \in \C^*.$$

It turns out that $h_r$ is completely determined by $(r,
\{\xi_\pi\}_{\pi \in
\Pi})$, and in particular is independent of the ideal
triangulation $\lambda$. This could be proved using Proposition~\ref{prop:PleatedCoordChang}, but here is a more geometric interpretation which also explains the terminology. Since the fundamental
group $\pi_1(S)$ is free, we can lift 
$r\col
\pi_1(S)
\to
\Isom(\HH^3) = \mathrm{PSL}_2(\C) $ to a group homomorphism
$\hat r \col
\pi_1(S) \to \mathrm{SL}_2(\C)$. Pick a small loop
$\gamma_k$ in $S$ going counterclockwise around the
$k$--th puncture
$v_k$. For an arbitrary choice of base point, the
class of $\gamma_k$ in $\pi_1(S)$ generates a peripheral
subgroup
$\pi_k\in\Pi$. In particular, $r(\gamma_k)$ fixes the point
$\xi_{\pi_k} \in \partial _\infty \HH^3 = \mathbb{CP}^1$
specified by the enhancement of $r$. As a consequence
$\xi_{\pi_k}\in \mathbb{CP}^1$, considered as a complex line in $\C^2$,  is contained in an eigenspace of
$\hat r(\gamma_k) \in \mathrm{SL}_2(\C)$ corresponding to
some eigenvalue $a_k\in \C$. 

\begin{lem}
\label{lem:PeriphLoad}
The total peripheral load of the homomorphism $(r,
\{\xi_\pi\}_{\pi \in
\Pi})$ from $\pi_1(S)$ to  $\Isom(\HH^3) = \mathrm{PSL}_2(\C) $  is equal to
 $$h_r = (-1)^p a_1^{-1}  a_2^{-1}  \dots a_p^{-1} $$ 
where, for an arbitrary lift $\hat r$ of $r$ to
$\mathrm{SL}_2(\C)$, the complex numbers
$a_k$ are associated as above to the $p$ punctures $v_k$ of
the surface $S$. 
\end{lem}

\begin{proof}
By definition of $a_k$, the derivative of $r(\gamma_k)$ at
the fixed point $\xi_{\pi_k}\in \mathbb{CP}^1$ is just  the complex multiplication
by $a_k^{2}$. If $x_{i_1}$, $x_{i_2}$, \dots, $x_{i_l}$ are
the shear-bend parameters of the edges of $\lambda$ that
are adjacent to the puncture $v_k$, counted with
multiplicities when both ends of an edge lead to $v_k$, it
follows from the definition of the shear-bend parameters
that this derivative is also equal to $x_{i_1}^{-1}x_{i_2}^{-1}
\dots x_{i_l}^{-1}$. Therefore, $x_{i_1}x_{i_2}
\dots x_{i_l} = a_k^{-2} $.

Considering all punctures, this proves that $$a_1^{-2} a_2^{-2} \dots a_p^{-2} = x_1^{2}
x_2^{2} \dots x_n^{2} = h_r^{2}$$ since every
edge of $\lambda$ has two ends. Therefore $h_r = \pm a_1^{-1} a_2^{-1} 
\dots a_p^{-1} $. 

We now  only need to determine the sign $\pm$. By
continuity, this sign is independent of the shear-bend
parameters $x_1$,
$x_2$, \dots, $x_n$ defining the enhanced homomorphism
$r$. We can therefore consider the case where the $x_i$ are
positive real. In this case, $r$ has discrete
image and, after conjugation, is valued
in $\mathrm{PSL}_2(\R) = \Isom(\HH^2)$. Pick a
decomposition of the surface $S$ into pairs of pants. For
each pair of pants $P$ in the decomposition, consider three
closed curves $\alpha_1$, $\alpha_2$ and
$\alpha_3$ going around the punctures of the interior of
$P$; an easy computation shows that exactly 0 or 2 of their
images 
$\hat r(\alpha_i) \in \mathrm{SL}_2(\R)$ have positive
traces (since the space of hyperbolic pairs of pants is connected, it suffices to perform the computation on a single example). By induction on the number of pairs of pants in the
decomposition of $S$, we conclude that the number of
$\hat r(\gamma_k)$ with positive trace is even. In other
words, the number of positive $a_k$ is even. Since $h_r$ is
positive in this case, it follows that the sign
$\pm$ is  equal to $(-1)^p$. 
\end{proof}

Combining the purely algebraic Proposition~\ref{prop:ClassifyLocalRep} and the purely geometric Proposition~\ref{prop:Pleated}, we obtain the
following rephrasing of the classification of local
representations, which combines the two points of view.

\begin{prop}
\label{prop:LocalRepEnhanced}
Let $\lambda$ be an ideal triangulation for the surface
$S$, with empty boundary.
Up to isomorphism of local representations, a local
representation
$
\rho\col \T_\lambda^q \to  \End (V_1 \otimes \dots
\otimes V_m)
$
is classified by  an enhanced homomorphism $(r,
\{\xi_\pi\}_{\pi \in
\Pi})$ from
$\pi_1(S)$ to
$\Isom(\HH^3)$ realizing $\lambda$, together with an $N$--th root $h=
\sqrt[N]{h_r}$ of the total peripheral load $h_r$ of $(r,
\{\xi_\pi\}_{\pi \in
\Pi})$. 

Conversely, any such  enhanced homomorphism $(r,
\{\xi_\pi\}_{\pi \in
\Pi})$ realizing $\lambda$ and any $N$--th root $h=
\sqrt[N]{h_r}$ of its total peripheral load $h_r$ are associated to a
local representation of $\T_\lambda^q$.
\qed
\end{prop}

Similarly,  the combination of Theorems~\ref{thm:LocalAndWeights} and \ref{thm:EnhHomClass} gives:

\begin{thm}
\label{thm:LocalRepGroupHom}
Let $q$ be a primitive $N$--th root of $(-1)^{N+1}$. 
There is a one-to-one correspondence between:
\begin{enumerate}
\item local representations $\rho = \left\{ 
\rho_\lambda\col \T_\lambda^q \to \End(V_1\otimes \dots V_m) \right\}_{\lambda \in \Lambda(S)}$ of the quantum Teichm\"uller space $\T_S^q$, considered up to isomorphisms of local representations of $\T_S^q$;
\item enhanced group homomorphisms $\bigl (r, \{\xi_\pi\}_{\pi \in \Pi} \bigr)$ from $\pi_1(S)$ to $\Isom(\HH^3)$, considered up to conjugation in $\Isom(\HH^3)$, together with the choice of an $N$--th root $h\in \C^*$ for the total peripheral load $h_r$ of $r$. \qed
\end{enumerate}
\end{thm}

\section{Intertwining operators}
\label{sect:Intertwiners}

Let $\rho = \left\{ 
\rho_\lambda\col \T_\lambda^q \to \End(V_\lambda) \right\}_{\lambda \in \Lambda(S)}$ and $\rho' = \left\{ 
\rho'_\lambda\col \T_\lambda^q \to \End(V_\lambda') \right\}_{\lambda \in \Lambda(S)}$ be two local representation of the quantum Teichm\"uller space $\T_S^q$ which are isomorphic. For instance, we can have $\rho'=\rho$. 

Recall that each space $V_\lambda$ splits as a tensor product $V_\lambda= V_1\otimes \dots V_m$, where each $V_j$ is an $N$--dimensional vector space attached to the $j$--th face $T_j$ of $\lambda$; similarly, $V_\lambda'$ decomposes as $V_\lambda'= V_1'\otimes \dots V_m'$. By definition, for every ideal triangulations $\lambda$ and $\lambda'$, the representation $\rho_\lambda  \circ \Phi_{\lambda\lambda'} \col \T_\lambda^q \to \End(V_\lambda)$ is isomorphic to  $\rho_{\lambda'}\col \T_{\lambda'}^q \to \End(V_{\lambda'})$, which itself is isomorphic to $\rho_{\lambda'}'\col \T_{\lambda'}^q \to \End(V'_{\lambda'})$. Therefore, there exists  a linear isomorphism $L_{\lambda\lambda'}^{\rho\rho'} \col V_{\lambda'}' \to  V_\lambda$ such that 
$$\rho_\lambda \circ \Phi_{\lambda\lambda'}(P') = L_{\lambda\lambda'}^{\rho\rho'}\, \,\rho'_{\lambda'} (P')  \,(L_{\lambda\lambda'}^{\rho\rho'})^{-1}$$
in $\End(V_\lambda)$, for every $P'\in \T_{\lambda'}^q$.

By definition, these isomorphisms $L_{\lambda\lambda'}^{\rho\rho'}$ are \emph{intertwining operators} between the isomorphic local representation $\rho$ and $\rho'$ of the quantum Teichm\"uller space.

Note that multiplying $L_{\lambda\lambda'}^{\rho\rho'}$ by a constant does not change the above property, so that the intertwining operators can be determined at most up to multiplication by a scalar. 
We will write $A \dot =B $ to say that the linear maps $A$ and $B$ are equal up to multiplication by a scalar. 

Actually, because local representations may have many irreducible factors, the intertwining operators cannot be unique even up to scalar multiplication unless we impose several additional, but natural, conditions. Defining a unique family of intertwining operators will require to consider all possible surfaces $S$ at the same time.

\begin{thm}
\label{thm:IntertwiningUnique}
Up to scalar multiplication, there exists a unique family of intertwining operators $L_{\lambda\lambda'}^{\rho\rho'}$, indexed by pairs of isomorphic local representations $\rho$ and $\rho'$ of quantum Teichm\"uller spaces $\T_S^q$ of the same surface $S$ and by ideal triangulations $\lambda$, $\lambda'$ of $S$, such that:
\begin{enumerate}

\item \emph{(Composition Relation)} for any three ideal triangulations $\lambda$, $\lambda'$, $\lambda''$ of the same surface $S$ and for any three isomorphic local representation $\rho$, $\rho'$, $\rho''$  of $\T_S^q$, we have that 
$ L_{\lambda\lambda''}^{\rho\rho''} \dot =L_{\lambda\lambda'}^{\rho\rho'} \circ L_{\lambda'\lambda''}^{\rho'\rho''}$;

\item \emph{(Fusion Relation)}  if $S$ is obtained from another surface $S_0$ by fusing $S_0$ along certain components of $\partial S_0$, if the representations $\rho$, $\rho'$ of the quantum Teichm\"uller space $\T_S^q$ is obtained by fusing isomorphic representations $\rho_0$, $\rho'_0$ of $\T_{S_0}^q$, and if $\lambda$, $\lambda'$ are two ideal triangulations of $S$ obtained by fusing ideal triangulations $\lambda_0$, $\lambda_0'$ of $S_0$, then $L_{\lambda\lambda'}^{\rho\rho'} \dot= L_{\lambda_0\lambda_0'}^{\rho_0\rho'_0}$. \end{enumerate}
\end{thm}

For the Fusion Relation (3), recall that, when the local representation $\rho_\lambda\col \T_\lambda \to \End(V_\lambda)$ is obtained by fusing the local representation $\rho_{\lambda_0}\col \T_{\lambda_0} \to \End(V_{\lambda_0})$, the two vector spaces $V_\lambda$ and $V_{\lambda_0}$ are equal. 

We will split the proof of Theorem~\ref{thm:IntertwiningUnique} into several steps.

We first restrict attention to the case of surfaces $S$ whose components  are polygons, namely closed disks with punctures removed from their boundary (and no internal puncture). Namely, we consider the restricted version of Theorem~\ref{thm:IntertwiningUnique} where all surfaces $S$ considered are disjoint union of polygons. 

\begin{lem}
\label{lem:IntertwinersPolygons}
The conclusions of Theorem~\ref{thm:IntertwiningUnique} hold if we restrict attention to surfaces $S$ which are disjoint union of polygons. 
\end{lem}

\begin{proof}
The key property in this case is that, if $S$ is a disjoint union of  polygons, any  local representation $\rho_\lambda\col \T_\lambda^q \to \End(V_\lambda) $ is irreducible. Indeed, if $S$ consists of $c$ polygons with a total of $p$ punctures,  one easily checks by induction on $p$ that $\T_\lambda^q$ is isomorphic to the algebra $\mathcal A_{p,c}$ defined by generators $A_1^{\pm1}$, $B_1^{\pm1}$, $A_2^{\pm1}$, $B_2^{\pm1}$, \dots, $A_{p-2c}^{\pm1}$, $B_{p-2c}^{\pm1}$, $H_1^{\pm1}$, $H_2^{\pm1}$, \dots, $H_c^{\pm1}$ and by the relations  $A_iB_i = q^2 B_iA_i$, $A_iB_j=B_jA_j$ if $i\not=j$, $A_iH_k = H_kA_i$, $B_iH_k=H_kB_i$. Very elementary algebra (compare  \cite[Lem.~18--19]{BonLiu}) then shows that every irreducible representation of $\T_\lambda^q\cong \mathcal A_{p,c}$ has dimension $N^{p-2c}$. This is also  the dimension of any of its local representations. It follows that every local representation $\rho_\lambda\col \T_\lambda^q \to \End(V_\lambda) $ is irreducible. 

Since $\rho_\lambda$ and $\rho_{\lambda'}' $ are irreducible, it is now immediate that the isomorphism $L_{\lambda\lambda'}^{\rho\rho'}$ between $\rho_\lambda$ and $\rho_{\lambda'}' $ is unique up to scalar multiplication. 

This uniqueness property implies that the intertwining operators satisfy the Composition and Fusion Relations.
\end{proof}

We now begin to prove the uniqueness part of Theorem~\ref{thm:IntertwiningUnique} in the general case. Namely, suppose that we are given a family of intertwiners $L_{\lambda\lambda'}^{\rho\rho'}$  as in Theorem~\ref{thm:IntertwiningUnique}. We will progressively show that these are uniquely determined up to scalar multiplication. 

We first consider the case where $\lambda'=\lambda$. 

\begin{lem}
\label{lem:IntertwSameTriang}
The intertwiners $L_{\lambda\lambda}^{\rho\rho'}$ are uniquely determined. 
\end{lem}

\begin{proof}
Let $T_1$, \dots, $T_m$ be the faces of $\lambda$. By definition of local representations, there are representations $\rho_j\col \T_{T_j}^q \to \End(V_j)$ of the associated triangle algebras such that the local representation $\rho_\lambda$ is represented by the triangle representations $\rho_1$, \dots $\rho_m$ if we formally consider a local representation as an equivalence class of $m$ representations of the triangle algebra; in practice, we usually abbreviate this fact as $\rho_\lambda=  \rho_1 \otimes \dots \otimes \rho_m$. 

Each representation $\rho_j$ is classified by weights $x_{ji}\in \C^*$ associated to the sides of the triangle $T_j$, and by a central load $h_j$. Similarly, $\rho_\lambda$ is classified by edge weights $x_i = x_{ji}x_{ki}$, where the edge $\lambda_i$ separates the two triangles $T_j$ and $T_k$, and by the central load $h=h_1\dots h_m$. The local representation $\rho_\lambda'$ is isomorphic to $\rho_\lambda$ and, by Corollary~\ref{cor:LocalIsom} is consequently classified by the same edge weights and the same central load. As a consequence, we can represent $\rho_\lambda'$ by a family of representations $\rho_j'\col \T_{T_j'}^q \to \End(V_j')$ such that each $\rho_j'$ is classified by the same edge weights $x_{ji}$ and the same central load $h_j$ as $\rho_j$. In particular, there exists an isomorphism $L_j\col V_j' \to V_j$ between the representations $\rho_j$ and $\rho_j'$. 

We can now apply the Fusion Relation to conclude that $L_{\lambda\lambda}^{\rho\rho'} \dot= L_1 \otimes \dots \otimes L_m$. Since the isomorphisms $L_j$ are unique up to scalar multiplication by irreducibility of the $\rho_j$, this shows that $L_{\lambda\lambda}^{\rho\rho'}$ is uniquely determined up to scalar multiplication once we have chosen a realization of $\rho$ by representations  $\rho_j\col \T_{T_j}^q \to \End(V_j)$ of the triangle algebras associated to the faces of $\lambda$.

It remains to show that $L_{\lambda\lambda}^{\rho\rho'}$ is independent of this choice. Other choices differ from this one by rescaling the images  $\rho_j(X_{ji})$ of the generators $X_{ji}$ of $\T_{T_j}^q$. Such a rescaling leads to the same rescaling of the corresponding $\rho_j'(X_{ji})$, so that we can keep the same $L_j$. This concludes the proof of Lemma~\ref{lem:IntertwSameTriang}. 
\end{proof}

We next consider the case where  $\lambda$ and $\lambda'$ differ only by a diagonal exchange. 

\begin{lem}
\label{lem:IntertwDiagEx}
If $\lambda$ and $\lambda'$ differ only by a diagonal exchange, then the intertwiner $L_{\lambda\lambda'}^{\rho\rho'}$ is uniquely determined up to scalar multiplication. 
\end{lem}
\begin{proof}
Split the surface $S$ along all the components of $\lambda$ except for the diagonal where the diagonal exchange takes place. One obtains a square $Q$ and $m-2$ triangles $T_3$, $T_4$, \dots, $T_{m}$, assuming without loss of generality that $T_1$ and $T_2$ are the two faces of $\lambda$ that are contained in $Q$.  The ideal triangulations $\lambda$, $\lambda'$ respectively give ideal triangulations $\lambda_Q$, $\lambda_Q'$ of $Q$, differing by a different choice of a diagonal for $Q$.  

The local representation $\rho_{\lambda}$ is obtained by fusing together a local representation $\rho_{\lambda_Q}\col \T_{\lambda_Q} \to \End(V_1\otimes V_2)$ and  representations $\rho_j\col \T_{T_j} \to \End(V_j)$ of the triangle algebras corresponding to $T_3$, \dots, $T_{m}$. Since the quantum coordinate changes satisfy the Fusion Relation of Section~\ref{subsect:Switch}, $\rho_\lambda \circ \Phi_{\lambda\lambda'}^q$ is obtained by fusing together the representation $\rho_{\lambda_Q} \circ \Phi_{\lambda_Q\lambda_Q'}^q$ of $\T_{\lambda_Q'}^q$ and the representations $\rho_j$ of the triangle algebras $\T_{T_j}^q$.

By Proposition~\ref{prop:WeightsGiveRep}, 
$\rho_{\lambda_Q} \circ \Phi_{\lambda_Q\lambda_Q'}^q$ is isomorphic to a local representation $\rho''_{\lambda_Q''}\col \T_{\lambda_Q'} \to \End(V''_1 \otimes V''_2)$.  If we set $V''_{\lambda'} = V''_1\otimes V''_2 \otimes V_3 \otimes \dots \otimes V_m$, let  $\rho''_{\lambda'}\col \T_{\lambda'}^q \to \End(V''_{\lambda'}) $ be the representation obtained by fusing together $\rho''_{\lambda'_Q}$ and the $\rho_j$ with $j\geq 3$. Extend $\rho''_{\lambda'}$ to a representation $\rho'' = \left\{ 
\rho''_\mu\col \T_\mu^q \to \End(W_\mu) \right\}_{\mu \in \Lambda(S)}$ of the quantum Teichm\"uller space by the property that $\rho''_\mu = \rho_\mu$ if $\mu \not= \lambda'$. 
Another application of  Proposition~\ref{prop:WeightsGiveRep} shows that the representations $\rho'$ and $\rho''$ of the quantum Teichm\"uller space are isomorphic, by a family of intertwining operators $L_{\mu\mu'}^{\rho\rho''}$.

If $L_{\lambda_Q\lambda_Q'} \col V_1'' \otimes V_2'' \to V_1 \otimes V_2$ is the isomorphism between  $\rho_{\lambda_Q} \circ \Phi_{\lambda_Q\lambda_Q'}^q$ and $\rho''_{\lambda_Q'}$, the Fusion Relation shows that modulo scalar multiplication  $L_{\lambda\lambda'}^{\rho\rho''}$ is just the tensor product of $L_{\lambda_Q\lambda_Q'} $ and of the identity maps of the $V_j$ with $j\geq 2$. 
By irreducibility of $\rho_{\lambda_Q}$ (or Lemma~\ref{lem:IntertwinersPolygons}), the isomorphism 
$L_{\lambda_Q\lambda_Q'} $ is unique up to scalar multiplication. It follows that $L_{\lambda\lambda'}^{\rho\rho''}$ is also unique up to scalar multiplication. 

Finally, the Composition Relation shows that $L_{\lambda\lambda'}^{\rho\rho'} \dot = L_{\lambda\lambda'}^{\rho\rho''}\circ L_{\lambda'\lambda'}^{\rho''\rho'}$. Applying the uniqueness property of Lemma~\ref{lem:IntertwSameTriang} to $L_{\lambda'\lambda'}^{\rho''\rho'}$, we conclude that $L_{\lambda\lambda'}^{\rho\rho'} $ is uniquely determined up to scalar multiplication.

However, we still have to check that this $L_{\lambda\lambda'}^{\rho\rho'} $ does not depend on the way we split $\rho_\lambda$ as the fusion of a local representation $\rho_{\lambda_Q}$ of $\T_{\lambda_Q}^q$ and of representations $\rho_j$ of $\T_{T_j}^q$ with $j\geq 3$. The argument is identical to the one we already used in the proof of Lemma~\ref{lem:IntertwSameTriang}.
\end{proof}

We are now ready to conclude the uniqueness part of Theorem~\ref{thm:IntertwiningUnique}. 

\begin{lem}
\label{lem:IntertwUnique}
The intertwiners $L_{\lambda\lambda'}^{\rho\rho'} $  are uniquely determined, up to scalar multiplication. 
\end{lem}

\begin{proof}
Any two ideal triangulations $\lambda$ and $\lambda'$ of the surface $S$ can by joined by a sequence of ideal triangulations $\lambda = \lambda^{(0)}$, $\lambda^{(1)}$, \dots, $\lambda^{(l-1)}$, $\lambda^{(l)}=\lambda'$ such that each $\lambda^{(k)}$ is obtained from $\lambda^{(k+1)}$  by a diagonal exchange. See \cite{Har, Pen, Hat}. By the Composition Relation, we necessarily have that 
$$
L_{\lambda\lambda'}^{\rho\rho'} \dot = L_{\lambda\lambda^{(1)}}^{\rho\rho} \circ L_{\lambda^{(1)}\lambda^{(2)}}^{\rho\rho} \circ \dots \circ L_{\lambda^{(l-1)}\lambda'}^{\rho\rho} \circ L_{\lambda'\lambda'}^{\rho\rho'} .$$
By Lemma~\ref{lem:IntertwDiagEx} and \ref{lem:IntertwSameTriang}, each $L_{\lambda^{(k)}\lambda^{(k+1)}}^{\rho\rho} $ and $L_{\lambda'\lambda'}^{\rho\rho'} $ are uniquely determined. Therefore, $L_{\lambda\lambda'}^{\rho\rho'} $ is uniquely determined, up to scalar multiplication. 
\end{proof}

We now prove the existence part of Theorem~\ref{thm:IntertwiningUnique}. The proof of the uniqueness tells us how to proceed. 

We first focus on the case where $\rho'=\rho$. 

Connect $\lambda$ to $\lambda'$ by a sequence of ideal triangulations $\lambda = \lambda^{(0)}$, $\lambda^{(1)}$, \dots, $\lambda^{(l-1)}$, $\lambda^{(l)}=\lambda'$ such that each $\lambda^{(k)}$ is obtained from $\lambda^{(k+1)}$  by a diagonal exchange.

 Let $L_{\lambda^{(k)}\lambda^{(k+1)}}^{\rho\rho}$ be the intertwining operator between $\rho_{\lambda^{(k)}} \col \T_{\lambda^{(k)}}^q \to \End(V_{\lambda^{(k)}})$ and $\rho_{\lambda^{(k+1)}} \col \T_{\lambda^{(k+1)}}^q \to \End(V_{\lambda^{(k+1)}})$ constructed in the proof of Lemma~\ref{lem:IntertwDiagEx}. The main property we need is the following: Index the faces of $\lambda^{(k)}$ and $\lambda^{(k+1)}$ so that the first two faces of each are located in the square where the diagonal exchange takes place, and so that for $j>2$ the $j$--th face of $\lambda^{(k)}$ is equal to the $j$--th face of  $\lambda^{(k=1)}$; then, for the corresponding splittings $V_{\lambda^{(k)}} = V_1 \otimes \dots \otimes V_m$ and $V_{\lambda^{(k+1)}} = W_1 \otimes \dots \otimes W_m$, $L_{\lambda^{(k)}\lambda^{(k+1)}}^{\rho\rho}$  is the tensor product of isomorphisms $W_1 \otimes W_2 \to V_1 \otimes V_2$ and $W_j \to V_j$ for $j>2$. Namely, $L_{\lambda^{(k)}\lambda^{(k+1)}}^{\rho\rho}$  satisfy a Fusion Relation when we split $S$ along all edges of $\lambda^{(k)}$ other than the diagonal being exchanged.  
 
 Define
 $$
 L_{\lambda\lambda'}^{\rho\rho} = L_{\lambda^{(0)}\lambda^{(1)}}^{\rho\rho} \circ L_{\lambda^{(1)}\lambda^{(2)}}^{\rho\rho}\circ \dots \circ L_{\lambda^{(l-1)}\lambda^{(l)}}^{\rho\rho}.
 $$

\begin{lem}
\label{lem:IntertwWellDefined}
The above intertwining operator $L_{\lambda\lambda'}^{\rho\rho} $ does not depend on the way we connect $\lambda$ to $\lambda'$ by a sequence of diagonal exchanges. 
\end{lem}

\begin{proof}
By a result of J. Harer \cite{Har} and R. Penner \cite{Pen}, any two such sequences
$\lambda=\lambda^{(0)}$,
$\lambda^{(1)}$, \ldots, $\lambda^{(l)}
=\lambda'$ and
$\lambda=\lambda'^{(0)}$,
$\lambda'^{(1)}$, \ldots, $\lambda'^{(l')}
=\lambda'$ of diagonal exchanges  can be
related to each other by successive applications of
the following moves and of their inverses.
\begin{enumerate}

\item (Round-trip Move) if $\lambda^{(k-1)}=\lambda^{(k+1)}$,
replace \dots, $\lambda^{(k-2)}$, $\lambda^{(k-1)}$, 
$\lambda^{(k)}$, $\lambda^{(k+1)}$, $\lambda^{(k+2)}$,
\dots\ by \dots, $\lambda^{(k-2)}$, $\lambda^{(k-1)}$, $\lambda^{(k+2)}$,
\dots;

\item (Distant Commutativity Move) if the consecutive diagonal exchanges in   \dots, 
$\lambda^{(k-1)}$, $\lambda^{(k)}$, $\lambda^{(k+1)}$,
\dots\ occur in squares $Q_{k-1}$, $Q_k$ whose interiors are disjoint, replace these two diagonal exchanges by  \dots, 
$\lambda^{(k-1)}$, $\mu^{(k)}$, $\lambda^{(k+1)}$, where $\mu^{(k)}$ is obtained from $\lambda^{(k-1)}$ by performing a diagonal exchange on the second square $Q_{k+1}$ (so that $\lambda^{(k+1)}$ is obtained from $\mu^{(k)}$ by performing a diagonal exchange in $Q_k$);

\item (Pentagon Move)
replace \dots, $\lambda^{(k)}$, \dots\
by
\dots
$\lambda^{(k)}=\mu^{(1)}$, $\mu^{(2)}$, $\mu^{(3)}$, $\mu^{(4)}$, $\mu^{(5)}$, $\mu^{(6)}=\lambda^{(k)}$, \dots\ where each $\mu^{(i)}$ is obtained from $\mu^{(i-1)}$ by a diagonal exchange in a pentagon as in Figure~\ref{fig:Pentagon}. 
\end{enumerate}

\begin{figure}[htb]
\SetLabels
(  0.13*  0.53)   $\mu^{(1)}  $ \\
( 0.48 * 0.88 )   $\mu^{(2)}  $ \\
(  0.9*  0.56)   $\mu^{(3)}  $ \\
( 0.7 *  0.09)   $\mu^{(4)}  $ \\
(  0.3* 0.15 )   $\mu^{(5)}  $ \\
\endSetLabels
\centerline{\AffixLabels
{\includegraphics[width=8cm]{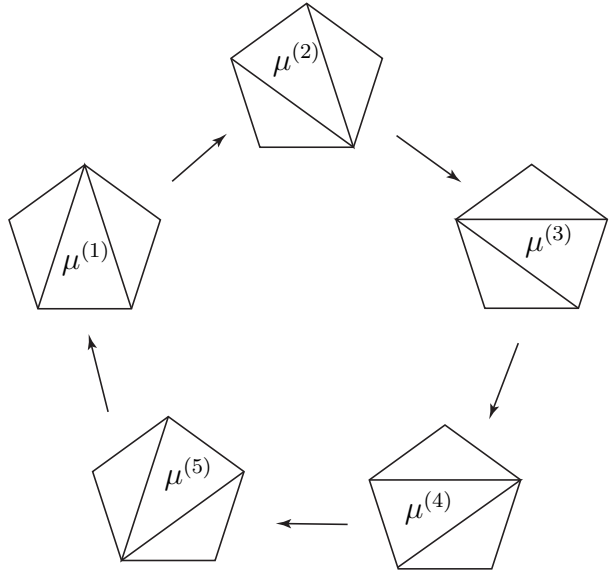}}}
\caption{The Pentagon Move}
\label{fig:Pentagon}
\end{figure}

 We consequently need to show that the $L_{\lambda\lambda'}^{\rho\rho} $ remains unchanged under these moves.
 
 Consider the Pentagon Move, with notation as above. In particular, the Pentagon Move inserts a term $L_{\mu^{(1)}\mu^{(2)}}^{\rho\rho} \circ L_{\mu^{(2)}\mu^{(3)}}^{\rho\rho}\circ L_{\mu^{(3)}\mu^{(4)}}^{\rho\rho}\circ L_{\mu^{(4)}\mu^{(5)}}^{\rho\rho}\circ L_{\mu^{(5)}\mu^{(6)}}^{\rho\rho}$  into the expression giving $L_{\lambda\lambda'}^{\rho\rho} $.

 Let $P$ be the pentagon where this move takes place. The  surface $S'$ obtained by splitting $S$ along those two edges of $\mu^{(1)}=\lambda^{(k)}$ that are not in the interior of $P$ consists of $P$ and of $n-3$ disjoint triangles $T_4$, $T_5$, \dots, $T_n$, corresponding to those faces of $\mu^{(1)}$ that are not in $P$. Note that each $\mu^{(i)} $ induces an ideal triangulation $\mu^{(i)}_{S'}$ of $S'$. 
 
 By our observation above Lemma~\ref{lem:IntertwWellDefined}, we can split  each of the representations $\rho_{\mu^{(i)}} \col \T_{\mu^{(i)}}^q \to \End(V_{\mu^{(i)}})$ as a fusion of representations $\rho_{\mu^{(i)}_{S'}} \col \T_{\mu^{(i)}_{S'}}^q \to \End(V_{\mu^{(i)}})$ in such a way that each each $\rho_{\mu^{(i)}_{S'}} \circ \Phi_{\mu^{(i)}_{S'}\mu^{(i+1)}_{S'}}^q$ is isomorphic to $\rho_{\mu^{(i+1)}_{S'}} $ by the isomorphism $L_{\mu^{(i)}\mu^{(i+1)}}^{\rho\rho} \col V_{\mu^{(i+1)}} \to V_{ \mu^{(i)}}$. 
 
 As a consequence, the composition $L_{\mu^{(1)}\mu^{(2)}}^{\rho\rho} \circ L_{\mu^{(2)}\mu^{(3)}}^{\rho\rho}\circ L_{\mu^{(3)}\mu^{(4)}}^{\rho\rho}\circ L_{\mu^{(4)}\mu^{(5)}}^{\rho\rho}\circ L_{\mu^{(5)}\mu^{(6)}}^{\rho\rho}$ is an isomorphism between $\rho_{\mu^{(6)}_{S'}}=\rho_{\lambda^{(k)}_{S'}}$ and 
\begin{equation*}
\begin{split}
&\rho_{\mu^{(1)}_{S'}} \circ \Phi_{\mu^{(1)}_{S'}\mu^{(2)}_{S'}}^q \circ \Phi_{\mu^{(2)}_{S'}\mu^{(3)}_{S'}}^q \circ \Phi_{\mu^{(3)}_{S'}\mu^{(4)}_{S'}}^q \circ \Phi_{\mu^{(4)}_{S'}\mu^{(5)}_{S'}}^q \circ \Phi_{\mu^{(5)}_{S'}\mu^{(6)}_{S'}}^q \\
= \,& \rho_{\mu^{(1)}_{S'}} \circ \Phi_{\mu^{(1)}_{S'}\mu^{(6)}_{S'}}^q =\rho_{\mu^{(1)}_{S'}} =\rho_{\lambda^{(k)}_{S'}}.
\end{split}
\end{equation*}

By our analysis of the case of disjoint union of polygons in Lemma~\ref{lem:IntertwinersPolygons}, the local representation $\rho_{\lambda^{(k)}_{S'}}$ is irreducible. Therefore, $L_{\mu^{(1)}\mu^{(2)}}^{\rho\rho} \circ L_{\mu^{(2)}\mu^{(3)}}^{\rho\rho}\circ L_{\mu^{(3)}\mu^{(4)}}^{\rho\rho}\circ L_{\mu^{(4)}\mu^{(5)}}^{\rho\rho}\circ L_{\mu^{(5)}\mu^{(6)}}^{\rho\rho}$ is a scalar multiple of the identity. 

This proves that the Pentagon Move does not change $L_{\lambda\lambda'}^{\rho\rho} $ up to scalar multiplication. 

The proof that the Roundtrip and Distant Commutativity Moves do not change $L_{\lambda\lambda'}^{\rho\rho} $ is essentially identical, replacing the pentagon $P$ by the one or two squares where these moves take place. 
\end{proof}

We have now defined $L_{\lambda\lambda'}^{\rho\rho} $ for every $\lambda$, $\lambda'$ and $\rho$, up to scalar multiplication. 

Having considered the case where $\rho'=\rho$, we now turn to the case where $\lambda'=\lambda$

If $\rho$ and $\rho'$ are isomorphic local representations of the quantum Teichm\"uller space, Corollary~\ref{cor:LocalIsom} shows that, for every ideal triangulation $\lambda$, we can choose the isomorphism $L_{\lambda\lambda}^{\rho\rho'} \col V_\lambda' \to V_\lambda$ between the corresponding local representations $\rho_\lambda\col \T_\lambda^q \to \End(V_\lambda)$ and  $\rho'_\lambda\col \T_\lambda^q \to \End(V_\lambda')$ to be tensor-split, namely to be an isomorphism of local representations. In addition, the proof of Lemma~\ref{lem:IntertwSameTriang} shows that this  $L_{\lambda\lambda}^{\rho\rho'}$ is unique up to scalar multiplication.  We henceforth make this choice for any such  $L_{\lambda\lambda}^{\rho\rho'}$.

We now combine the two cases.

\begin{lem}
\label{lem:IntertwCommute}
$$ L_{\lambda\lambda'}^{\rho\rho} \circ L_{\lambda'\lambda'}^{\rho'\rho'} \,\dot= \, L_{\lambda\lambda}^{\rho\rho'}  \circ  L_{\lambda\lambda'}^{\rho'\rho'}.$$
\end{lem}

\begin{proof}
It suffices to consider the case where $\lambda$ and $\lambda'$ differ by a diagonal exchange, in which case the result is immediate from definitions. 
\end{proof}

In the general case, we now define  $L_{\lambda\lambda'}^{\rho\rho'} =  L_{\lambda\lambda'}^{\rho\rho} \circ L_{\lambda'\lambda'}^{\rho'\rho'} \,\dot= \, L_{\lambda\lambda}^{\rho\rho'}  \circ  L_{\lambda\lambda'}^{\rho'\rho'} $ for every $\rho$, $\rho'$, $\lambda$, $\lambda'$.

To complete the proof of Theorem~\ref{thm:IntertwiningUnique}, we need to check that the $L_{\lambda\lambda'}^{\rho\rho'} $ just defined satisfy the Composition and Fusion Relations. These are automatic provided that we suitably choose the sequence of diagonal exchanges $\lambda = \lambda^{(0)}$, $\lambda^{(1)}$, \dots, $\lambda^{(l-1)}$, $\lambda^{(l)}=\lambda'$ connecting $\lambda$ to $\lambda'$, and using Lemma~\ref{lem:IntertwCommute}. 

This concludes our proof of Theorem~\ref{thm:IntertwiningUnique}. \qed

In the intertwiner $L_{\lambda\lambda'}^{\rho\rho'}$, a lot of the data is actually redundant. Indeed, the following lemma shows that one does not need to know the whole representations $\rho$ and $\rho'$ of the quantum Teichm\"uller space, only the parts that are relevant to $\lambda$ and $\lambda'$. (The extraneous data is however required to guarantee uniqueness in the statement of Theorem~\ref{thm:IntertwiningUnique}). 

\begin{lem}
\label{lem:IntertwLocallyDefined}
Up to scalar multiplication, the intertwiner $L_{\lambda\lambda'}^{\rho\rho'} \col V'_{\lambda' }\to V_\lambda$ provided by Theorem~\ref{thm:IntertwiningUnique} depends only on the ideal triangulations $\lambda$ and $\lambda'$  on the local representations  $\rho_\lambda\col \T_\lambda^q \to \End(V_\lambda)$ and  on $\rho'_{\lambda' }\col \T_{\lambda' }^q \to \End(V_{\lambda' }')$ induced by $\rho$ and $\rho'$. 
\end{lem}
\begin{proof}
Let $\rho''$ be another isomorphic representation, such that the induced local representation $\rho''_{\lambda'}$ coincides with $\rho'_{\lambda'}$, although $\rho''_{\mu}$ may be different from $\rho'_{\mu}$ when $\mu \not= \lambda'$. 

By the Composition Relation, $L_{\lambda\lambda'}^{\rho\rho''}  \dot = L_{\lambda\lambda'}^{\rho\rho'}  \circ L_{\lambda'\lambda'}^{\rho'\rho''}$. However, $ L_{\lambda'\lambda'}^{\rho'\rho''}$ is the identity by construction, so that $L_{\lambda\lambda'}^{\rho\rho''}  \dot = L_{\lambda\lambda'}^{\rho\rho'}  $.  

This proves that $L_{\lambda\lambda'}^{\rho\rho''}$ is independent of the $\rho'_{\mu}$ with $\mu \not= \lambda'$. 

By the same argument, it is also independent of the $\rho_{\mu}$ with $\mu \not= \lambda$. 
\end{proof}

\section{The Kashaev bundle}
\label{sect:Kashaev}

We now investigate a construction which was suggested to us by Rinat Kashaev. It will require to select $N$--th roots $x_i^{\frac1N}$ for the weights $x_i\in \C^*$ associated by an enhanced homomorphism to the edges of an ideal triangulation. 

\subsection{Choices of $N$--th roots}

In  the space of (conjugacy classes of) enhanced homomorphisms $(r,
\{\xi_\pi\}_{\pi \in
\Pi})$, we consider a subset $\widetilde\U$ with the following properties:
\begin{enumerate}

\item every $r\in \widetilde \U$ is peripherically generic, namely realizes every ideal triangulation $\lambda$;

\item for every edge $\lambda_i$ of an ideal triangulation $\lambda$ and for every enhanced homomorphism $r\in \widetilde \U$, we are given a preferred $N$--th root $x_i^{\frac1N}$ for the weight $x_i$ associated to $\lambda_i$ and $r$ by Proposition~\ref{prop:Pleated},  and this root $x_i^{\frac1N}$ depends continuously on $r$;

\item the product $x_1^{\frac1N}x_2^{\frac1N} \dots x_n^{\frac1N}$ of the $N$--th roots associated to $r\in \widetilde\U$ and to the edges of an ideal triangulation $\lambda$ depends only on $r$, not on $\lambda$; 

\item the space $\widetilde\U$ is invariant under the action of the mapping class group $\pi_0\, \Diff(S)$. 

\end{enumerate}

For Condition (3), recall that Lemma~\ref{lem:PeriphLoad} that the peripheral load $x_1x_2\dots x_n$ depends only on $r$, not on $\lambda$. For Condition (5), note that it is immediate from the construction of shear-bend coordinates that $y_j = x_i$.

There are several geometric examples that we have in mind.

\subsubsection*{The enhanced Teichm\"uller space}  The first one is the case where $\widetilde U$ is the \emph{enhanced Teichm\"uller space} $\T(S)$, consisting of those enhanced homomorphisms which are injective, have discrete image, and are valued in $\Isom (\HH^2) \subset \Isom(\HH^3)$.  In this case, the shear coordinates $x_i$ are always positive reals, so that there is a well defined positive real $N$--root $x_i^{\frac1N}$. This choice is natural enough that all conditions are automatically satisfied.

\subsubsection*{The cusped Teichm\"uller space} A subset of the previous case, which occurs in many geometric contexts,  is the \emph{cusped Teichm\"uller space} $\T_{\mathrm c}(S)$, consisting of those homomorphisms which have discrete image, are valued in $\Isom (\HH^2) \subset \Isom(\HH^3)$, and send each peripheral subgroup of $\pi_1(S)$ to a parabolic subgroup of $\Isom (\HH^2) $. The parabolicity condition implies that such a group homomorphism $r$ admits a unique enhancement. It also implies that the total peripheral load $h_r=x_1x_2\dots x_n$ is equal to 1, so that its positive real $N$--th root $x_1^{\frac1N} x_2^{\frac1N}\dots x_n^{\frac1N}$ is also equal to 1. 

\subsubsection*{The quasifuchsian space}
The cusped Teichm\"uller space is contained in a larger space $\mathcal {QF}(S)$ of \emph{quasifuchsian homomorphisms}. The space $\mathcal {QF}(S)$ is the interior of the space of all homomorphisms $r\col \pi_1(S) \to \Isom (\HH^3)$ with discrete image and sending peripheral subgroups of $\pi_1(S)$ to parabolic subgroups of $\Isom(\HH^3)$. As such, this space arises is much of hyperbolic geometry.  Again, a quasifuchsian homomorphism admits a unique enhancement. The space  $\mathcal {QF}(S)$ is simply connected, so that the choice of positive real $N$--th roots $x_i^{\frac1N}$ on $\T_{\mathrm c}(S)$ uniquely extends to a continuous choice of complex roots on $\mathcal{QF}(S)$. Since the required conditions are satisfied on the cusped Teichm\"uller space $\T_{\mathrm c}(S)$, they automatically hold over all of $\mathcal {QF}(S)$ by continuity. In addition, the $N$--th root $x_1^{\frac1N} x_2^{\frac1N}\dots x_n^{\frac1N}$ of the total peripheral load $h_r$ of $r\in\mathcal{QF}(S)$ is also equal to 1.

\subsubsection*{The intrinsic closure of the quasifuchsian space}
Finally, we can consider the \emph{intrinsic closure} $\overline{\mathcal{QF}}(S)$ of $\mathcal{QF}(S)$, defined as follows. The technique of pleated surfaces shows that the space of enhanced homomorphisms is a manifold on a neighborhood of the closure of $\mathcal{QF}(S)$. Endow this neighborhood with an arbitrary riemannian metric, and endow $\mathrm{QF}(S)$ with the metric for which the distance between $r$ and $r'$  is the infimum of the lengths of all curves joining $r$ to $r'$ and completely contained in $\mathcal{QF}(S)$. Then, $\overline{\mathcal{QF}}(S)$ is defined as the completion of $\mathcal{QF}(S)$ for this metric. One easily sees that, up to homeomorphism, this intrinsic completion is independent of our initial choice of a riemannian metric. This intrinsic closure is different from the usual closure because of the phenomenon of ``self-bumping'' \cite{McM, AndCan, AndCanMcC, CaMcC}. There is a natural projection from $\overline{\mathcal{QF}}(S)$ to the closure of $\mathcal{QF}(S)$, which is one-to-one at most points; indeed, self-bumping can only occur at points with accidental parabolics, corresponding to non-peripheral curves that are sent to parabolic elements. The definition is specially designed
 so that the choices of $N$--th roots $x_i^{\frac1N}$ over $\mathcal{QF}(S)$ continuously extend to this intrinsic closure $\overline{\mathcal{QF}}(S)$. Technically, $\overline{\mathcal{QF}}(S)$  does not quite fit within the framework mentioned at the beginning of this section, because it is more than just a subset of the space of enhanced homomorphism, but the abuse of language will be convenient. 
 
 \subsection{Explicit local representations}
 The $N$--th roots $x_i^{\frac1N}$ enable us to provide an explicit description of the local representation associated by Proposition~\ref{prop:LocalRepEnhanced} to an enhanced homomorphism $r\in \widetilde\U$ and to a choice of global load $h$. They also provide us with a uniform choice for this global load $h = \sqrt[N]{x_1x_2\dots x_n}$ by taking $h = q^{2k} x_i^{\frac1N}x_2^{\frac1N} \dots x_n^{\frac1N}$ for some $N$--th root of unity $q^{2k}$.
 
 First fix a root of unity $q^{2k}$. 
 
 Then, for every ideal triangulation $\lambda$, choose a local representation $\rho_\lambda\col \T_\lambda^q \to \End(V_\lambda)$ of its Chekhov-Fock algebra $\T_\lambda^q = \C[X_1^{\pm1}, X_2^{\pm1}, \dots, X_n^{\pm1}]_\lambda^q$ of $\lambda$ sending each $X_i^N$ to the identity  $\Id_{V_\lambda}$ and sending the principal central element $H$ to $q^{2k}\Id_{V_\lambda}$. Namely,  the local representation $\rho_\lambda$ is classified by edge weights all equal to 1 and by the central load $q^{2k}$. We will call such a representation \emph{$q^{2k}$--standard}. 
 
 Once we have made these choices,  we can define a local representation $\rho_{\lambda,r} \col \T_\lambda^q \to \End(V_\lambda)$ for every $r\in \widetilde\U$, as follows. Let $x_1$, $x_2$, \dots $x_n\in \C^*$ be the edge weights associated to $\lambda$ by $r$, and let $x_1^{\frac1N}$, $x_2^{\frac1N}$, \dots, $x_n^{\frac1N}$ be their $N$--th roots specified by $\widetilde U$. Then $\rho_{\lambda,r}$ is defined by the property that $\rho_{\lambda,r}(X_i) = x_i^{\frac1N}\rho_\lambda(X_i)$ for the generator $X_i$ associated to the $i$-th edge of $\lambda$. By construction, $\rho_{\lambda,r}$ is classified by the edge weights $x_i$ and by the central load $h = q^{2k} x_i^{\frac1N}x_2^{\frac1N} \dots x_n^{\frac1N}$, in the sense of Proposition~\ref{prop:ClassifyLocalRep}. Namely, in the sense of Proposition~\ref{prop:LocalRepEnhanced}, $\rho_{\lambda,r}$ is classified by the enhanced homomorphism $r\in \widetilde\U$ and by the $N$--th root $h = q^{2k} x_i^{\frac1N}x_2^{\frac1N} \dots x_n^{\frac1N}$ of its total peripheral load $h_r=x_1x_2\dots x_n$. 
 
 The construction is specially designed that, for every enhanced homomorphism $r\in \widetilde\U$, the family
 $$
\rho_r = \bigl \{ \rho_{\lambda,r} \col \T_\lambda^q \to \End(V_\lambda) ) \bigr\}_{\lambda \in \Lambda(S)}
 $$
 forms a representation of the quantum Teichm\"uller space, classified by the enhanced homomorphism $r\in \widetilde\U$ and by the $N$--th root $h = q^{2k} x_i^{\frac1N}x_2^{\frac1N} \dots x_n^{\frac1N}$ of its total peripheral load $h_r=x_1x_2\dots x_n$.

 \subsection{The Kashaev bundle} 
 
 In the construction of the previous section, suppose that we are given another family of $q^{2k}$--standard representations $\rho_\lambda' \col \T_\lambda^q \to \End(V_\lambda')$,  classified by edge weights $x_i=1$ and central load $h = q^{2k}$.  We then get a new representation $\rho_r'$ of the quantum Teichm\"uller space. Theorem~\ref{thm:IntertwiningUnique} provides a family of intertwining operators $L_{\lambda\lambda'}^{\rho_r\rho_r'} \col V_{\lambda'}' \to V_\lambda$. These intertwining operators are only defined up to scalar multiplication, but they induce a uniquely determined projective isomorphism $\PP( L_{\lambda\lambda'}^{\rho_r\rho_r'}) \col \PP(V_{\lambda'}' )\to \PP(V_\lambda)$. We can then interpret the Composition Relation in Theorem~\ref{thm:IntertwiningUnique} as a cocycle condition defining a fiber bundle over $\widetilde\U$.

 More precisely, given an $N$--th root of unity $q^{2k}$, define 
 $$
 \K_{q^{2k}} (\widetilde \U)= \bigsqcup_{\lambda,\, \rho_\lambda} \widetilde\U \times \PP(V_\lambda) / \sim
 $$
 where: 
 \begin{enumerate}
\item $\lambda$ ranges over all ideal triangulations of $S$; 

\item $\rho_\lambda$ ranges over  all $q^{2k}$--standard  local representation $\T_\lambda^q \to \End(V_\lambda)$ of the Chekhov-Fock algebra of $\lambda$; 

\item $\PP(V_\lambda)$ is the projective space of $V_\lambda$; 

\item the equivalence relation $\sim$ identifies $(r,v) \in \widetilde\U \times  \PP(V_\lambda)$ to $(r',v') \in \widetilde\U \times  \PP(V_{\lambda'}')$ exactly when $r=r'$ and $v= \PP( L_{\lambda\lambda'}^{\rho_r\rho_r'})(v')$, where $ \PP( L_{\lambda\lambda'}^{\rho_r\rho_r'})$ is the projectivized intertwining operator associated to any two local representations $\rho_r$ and $\rho_r'$ of the quantum Teichm\"uller space respectively containing $\rho_{\lambda,r}$ and $\rho_{\lambda',r}'$. 

\end{enumerate}
Note that $ \PP( L_{\lambda\lambda'}^{\rho_r\rho_r'})$ is independent of the choice of $\rho_r$ and $\rho_r'$ by Lemma~\ref{lem:IntertwLocallyDefined}, so that $\sim$ is well-defined. Also, the Composition Relation is exactly what is needed to make sure that $\sim$ is an equivalence relation, while the continuity of the roots $x_i^{\frac1N}$ guarantees that $ \PP( L_{\lambda\lambda'}^{\rho_r\rho_r'})$ depends continuously on $r$. 

There is a well defined map $\widetilde\pi\col \K_{q^{2k}} (\widetilde \U) \to \widetilde \U$ which associates $r\in \widetilde\U$ to the class of $(r,v) \in\widetilde \U \times \PP(V_\lambda)$.

If we fix an ideal triangulation $\lambda$ and a $q^{2k}$--standard  local representation $\rho_\lambda\col \T_\lambda^q \to \End(V_\lambda)$, the quotient map $ \widetilde\U \times \PP(V_\lambda) \to  \K_{q^{2k}} (\widetilde \U)$ is of course a homeomorphism. In particular:

\begin{prop}
$\widetilde\pi\col \K_{q^{2k}} (\widetilde \U) \to \widetilde \U$  is a trivial bundle, with fiber $\mathbb{CP}^{mN-1}$. 
\end{prop}

This may sound topologically uninteresting. 
However, the main point of the construction is that it is independent of any choice of $\lambda$ and $\rho_\lambda$. As a consequence:

\begin{prop}
The action of the mapping class group $\pi_0\, \Diff (S)$ on $\widetilde\U$ canonically lifts to an action on $ \K_{q^{2k}} (\widetilde \U)$. 
\end{prop}
 
 \begin{proof}
A diffeomorphism $\phi\col S \to S$ provides an identification $\Psi_{\phi, \lambda} \col \T_\lambda^q \to \T_{\phi(\lambda)}^q$ of the Chekhov-Fock algebras of $\lambda$ and $\phi(\lambda)$. 
Consequently, a $q^{2k}$--standard local representation  $\rho_\lambda\col \T_\lambda^q \to \End(V_\lambda)$ gives a $q^{2k}$--standard  local representation  $\rho'_\lambda\col \T_\lambda^q \to \End(V_\lambda')$ with $V_\lambda' = V_{\phi(\lambda)}$ and $\rho'_\lambda = \rho_{\phi(\lambda)} \circ \Psi_{\phi, \lambda} $.

The action is now clear. With the notation of the construction of $ \K_{q^{2k}}(\widetilde \U)$, the action of  $\phi\in \pi_0 \, \Diff(S)$ sends the element of $ \K_{q^{2k}} (\widetilde \U)$ represented by 
$(r,v)\in \widetilde\U \times \PP(V_\lambda)$ to the element represented by 
$$\left(r\circ \phi_*,  \PP( L_{\phi(\lambda)\lambda}^{\rho_r'\rho_r})(v)\right ) \in \widetilde \U \times V_{\phi(\lambda)}.$$ Here $\phi_*$ denotes both the homomorphism $ \pi_1(S) \to \pi_1(S)$ induced by $\phi$ and its action on the possible enhancements. 
\end{proof}

 \begin{thm}
If $\phi\in \pi_0 \, \Diff(S)$ fixes an element $r \in \widetilde\U$, the action of the lift $\K_{q^{2k}} (\widetilde \U) \to \K_{q^{2k}} (\widetilde \U)$ of $\phi$ on the fiber $\pi^{-1}(r)$ is uniquely determined up to projective transformation of this fiber. \qed
\end{thm}
 
 If we apply this to a pseudo-Anosov surface diffeomorphism and to the fixed point $r\in \overline {\mathcal{QF}}(S)$ provided by the complete hyperbolic metric of its mapping torus $M_\phi$, this is the analogue of the construction that we used in \cite{BonLiu}, replacing irreducible representations by local representations of the quantum Teichm\"uller space. If we represent this projective transformation by a matrix with determinant 1, its trace is equal modulo a root of unity to the invariant associated by Baseilhac and Benedetti \cite{BasBen05} to the hyperbolic metric of $M_\phi$. See \cite{Bon10}. 
 
Now, assume that,  in addition, the action of the mapping class group $\pi_0 \Diff (S)$ on $\widetilde\U$ is properly discontinuous. For instance, this holds when $\widetilde\U$ is the enhanced Teichm\"uller space $\T(S)$, or the cusped Teichm\"uller space $\T_{\mathrm c}(S)$, or the quasifuchsian space $\mathcal{QF}(S)$. We can then consider the quotient spaces $\U = \widetilde \U/ \pi_0 \,\Diff (S)$ and $\K_{q^{2k}} ( \U) = \K_{q^{2k}} (\widetilde \U) / \pi_0 \,\Diff (S)$. 

\begin{prop}
The map  $\pi\col K_{q^{2k}} ( \U) \to \U$ induced by $\widetilde\pi\col \K_{q^{2k}} (\widetilde \U) \to \widetilde \U$ is an orbifold bundle map, with fiber $\mathbb{CP}^{mN-1}$. \qed
\end{prop}

This bundle is non-trivial. For instance, there is an obstruction to lift a complex projective bundle to a vector bundle, which lies in the second cohomology group of the base with coefficients in the multiplicative. In the case of $\widetilde\pi\col \K_{q^{2k}} (\widetilde \U) \to \widetilde \U$, this obstruction is explicitly computed in \cite{BaiBon}, where it is also shown that it is non-trivial for $\widetilde\U = \T(S)$, $\T_{\mathrm c}(S)$ or $\mathcal{QF}(S)$.


\begin{thebibliography}{xx}

\bibitem{AndCan} James W. Anderson, Richard D.  Canary, \emph{Algebraic limits of Kleinian groups which rearrange the pages of a book},   Invent. Math.  126  (1996), 205--214. 

\bibitem{AndCanMcC} James W. Anderson, Richard D.  Canary, Darryl McCullough, \emph{The topology of deformation spaces of Kleinian groups},  Ann. of Math.  152  (2000),   693--741.

\bibitem{Bai1} Hua Bai,
\emph{A uniqueness property for the quantization of
Teichm\"uller spaces},
preprint, 2005  (\texttt{ArXiv:math.GT/0509679}).

\bibitem{Bai2} Hua Bai, \emph{Quantum Teichm\"uller spaces and Kashaev's 6j-symbols}, preprint, 2007, (\texttt
{ArXiv:0706.2026}). 

\bibitem{BaiBon} Hua Bai, Francis  Bonahon, \emph{Intertwiners for local representations
of the quantum Teichm\"uller space}, in preparation.

\bibitem{BasBen02} St\'ephane
Baseilhac, Riccardo Benedetti,
\emph{QHI, $3$--manifolds
scissors congruence and the
volume conjecture},
Geom. Topol. Monogr. 4 (2002),
13--28.

\bibitem{BasBen03} St\'ephane
Baseilhac, Riccardo Benedetti,
\emph{Quantum hyperbolic invariants
of $3$--manifolds with
$\mathrm{PSL}_2(\mathbb
C)$--characters},
Topology 43 (2004), 1373--1423. 


\bibitem{BasBen05} St\'ephane
Baseilhac, Riccardo Benedetti,
\emph{Classical and quantum
dilogarithmic invariants of flat
$\mathrm{PSL}_2(\mathbb
C)$--bundles over $3$--manifolds},
Geom.Topol. 9 (2005) 493--569.

\bibitem{BasBen06} St\'ephane
Baseilhac, Riccardo Benedetti,
\emph{Quantum hyperbolic geometry}, preprint, 2006,
(\texttt
{ArXiv:math.GT/0611504}).
 


\bibitem{Bon97} Francis Bonahon,
\emph{Shearing hyperbolic
surfaces, bending pleated
surfaces and Thurston's
symplectic form},  Ann. Fac. Sci.
Toulouse Math.   5  (1996), 
233--297. 

\bibitem{Bon10} Francis Bonahon, \emph{The invariants of Kashaev-Baseilhac-Benedetti}, in preparation. 


\bibitem{BonLiu} Francis Bonahon, Xiaobo Liu, \emph{Representations
of
the quantum Teichm\"uller
space and invariants of
surface diffeomorphisms}, Geom. Topol. 11 (2007), 889-938. 


\bibitem{CEG} Richard D. Canary,
David B. A.  Epstein, Paul
Green,  \emph{Notes on notes of
Thurston}, in:  \emph{Analytical
and geometric aspects of
hyperbolic space}
(Coventry/Durham, 1984),  3--92,
London Math. Soc. Lecture Note
Ser. vol. 111, Cambridge Univ.
Press, Cambridge, 1987.

\bibitem{CaMcC} Richard D. Canary, Darryl McCullough, \emph{Homotopy equivalences of $3$--manifolds and deformation theory of Kleinian groups},   Mem. Amer. Math. Soc.  172  (2004). 

\bibitem{CF} Leonid O.
Chekhov, Vladimir V. Fock,
 \emph{Quantum Teichm\"uller spaces},
Theor.Math.Phys. 120 (1999) 1245-1259. 

\bibitem{CF2} Leonid O.
Chekhov, Vladimir V. Fock,
 \emph{Observables in 3D gravity
and geodesic algebras}, in:
\emph{Quantum groups and
integrable systems} (Prague,
2000),  Czechoslovak J. Phys.  50 
(2000),   1201--1208.

\bibitem{Coh} Paul M. Cohn,
\emph{Skew Fields: Theory of
General Division Rings},
Encyclopedia of Mathematics and
its Applications 57, Cambridge
University Press, 1995. 

\bibitem{Fok} Vladimir V. Fock, 
 \emph{Dual Teichm\"uller
spaces},
preprint,
1997
(\texttt{ArXiv: Math/dg-ga/9702018})



\bibitem{Har} John L. Harer, 
\emph{The virtual cohomological
dimension of the mapping class
group of an orientable surface}, 
Invent. Math.  84  (1986),
157--176.



\bibitem{Hat} Allen E.  Hatcher, 
\emph{On triangulations of surfaces}, 
Topology Appl. 40 (1991), 189--194. 

\bibitem{Kash1} Rinat Kashaev, 
\emph{A link invariant from
quantum dilogarithm},  Modern
Phys. Lett. A  10  (1995), 
1409--1418. 

\bibitem{Kash2} Rinat Kashaev,
\emph{The hyperbolic volume of
knots from the quantum
dilogarithm},  Lett. Math. Phys. 
39  (1997),  no. 3, 269--275.

\bibitem{Kash3} Rinat Kashaev, 
\emph{Quantization of
Teichm\"uller spaces and the
quantum dilogarithm},  Lett. Math.
Phys.  43  (1998),  105--115.

\bibitem{Kas} Christian Kassel, 
\emph{Quantum groups}, Graduate
Texts in Mathematics vol. 155,
Springer-Verlag, New York, 1995. 


\bibitem{Liu1} Xiaobo Liu,
\emph{The quantum Teichm\"uller
space as a non-commutative
algebraic object}, preprint, 2004
(\texttt{ArXiv:math.GT/0408361}).

\bibitem{Liu2} Xiaobo Liu,
\emph{Quantum hyperbolic invariants
 for diffeomorphisms of small surfaces}, 
 preprint, 2005
(\texttt{ArXiv:math.GT/0603467}).

\bibitem{McM} Curtis T. McMullen, 
\emph{Complex earthquakes and Teichm\"uller theory}, 
J. Amer. Math. Soc. 11 (1998), 283--320. 

\bibitem{Pen} Robert C. Penner, 
\emph{The decorated Teichm\"uller
space of punctured surfaces}, 
Comm. Math. Phys.  113  (1987), 
299--339.


\bibitem{Thu} William
P. Thurston,
\emph{Three-dimensional
manifolds, Kleinian groups and
hyperbolic geometry},  Bull. Amer.
Math. Soc. 6  (1982), 
357--381. 








\end{thebibliography}
\end{document}